\documentclass[12pt]{amsart}
\usepackage{amsmath}
\usepackage{amsthm}
\usepackage{amssymb}
\usepackage{graphics}
\usepackage{bm}
\usepackage{setspace}
\usepackage{fullpage}

\usepackage{amsfonts}
\usepackage{eucal}
\usepackage{latexsym}
\usepackage{mathdots}
\usepackage{mathrsfs}
\usepackage{enumitem}

\usepackage{color}

\newcommand{\be}{\begin{equation}}
\newcommand{\ee}{\end{equation}}
\newcommand{\ba}{\begin{eqnarray}}
\newcommand{\ea}{\end{eqnarray}}
\newcommand{\bal}{\begin{align}}
\newcommand{\eal}{\end{align}}
\newcommand{\baln}{\begin{align*}}
\newcommand{\ealn}{\end{align*}}
\newcommand{\bi}{\begin{itemize}}
\newcommand{\ei}{\end{itemize}}
\newcommand{\bn}{\begin{enumerate}}
\newcommand{\en}{\end{enumerate}}
\newcommand{\bbm}{\begin{bmatrix}}
\newcommand{\ebm}{\end{bmatrix}}
\newcommand{\bpm}{\begin{pmatrix}}
\newcommand{\epm}{\end{pmatrix}}
\newcommand{\bp}{\begin{proof}}
\newcommand{\ep}{\end{proof}}
\newcommand{\nn}{\nonumber}

\newcommand{\mr}{\ensuremath{\mathrm}}
\newcommand{\scr}{\ensuremath{\mathscr}}
\newcommand{\mbf}{\ensuremath{\mathbf}}
\newcommand{\mc}{\ensuremath{\mathcal}}
\newcommand{\mf}{\ensuremath{\mathfrak}}
\newcommand{\ov}{\ensuremath{\overline}}
\newcommand{\sm}{\ensuremath{\setminus}}
\newcommand{\wt}{\ensuremath{\widetilde}}

\newcommand{\Ga}{\ensuremath{\Gamma}}
\newcommand{\ga}{\ensuremath{\gamma}}
\newcommand{\Om}{\ensuremath{\Omega}}

\newcommand{\om}{\ensuremath{\omega}}

\newcommand{\eps}{\ensuremath{\epsilon }}

\def\C{\mathbb{C}}

\def\D{\mathbb{D}}

\def\N{\mathbb{N}}
\def\B{\mathbb{B}}

\def\A{\mathcal{A} _d}

\def\fp{\mathbb{C} \{ \mathfrak{z} _1 , ..., \mathfrak{z} _d \} }

\newcommand{\cH}{\ensuremath{\mathcal{H}}}
\newcommand{\J}{\ensuremath{\mathcal{J} }}

\newcommand{\K}{\ensuremath{\mathcal{K} }}

\newcommand{\F}{\ensuremath{\mathbb{F} }}

\newcommand{\cK}{\ensuremath{\mathcal{K}}}
\newcommand{\cM}{\ensuremath{\mathcal{M}}}
\newcommand{\cN}{\ensuremath{\mathcal{N}}}

\newcommand{\bsm}{\left ( \begin{smallmatrix}}
\newcommand{\esm}{\end{smallmatrix} \right) }

\newcommand{\ip}[2]{\ensuremath{\langle {#1} , {#2} \rangle}}
\newcommand{\ipcn}[2]{\ensuremath{\left( {#1} , {#2} \right) _{\C ^n}}}
\newcommand{\ipcN}[2]{\ensuremath{\left( {#1} , {#2} \right) _{\C ^N}}}

\renewcommand{\dim}[1]{\ensuremath{\mathrm{dim} \left( {#1} \right) }}
\newcommand{\ran}[1]{\ensuremath{\mathrm{Ran} \left( {#1} \right) }}
\newcommand{\rank}[1]{\ensuremath{\mathrm{rank} \left( {#1} \right) }}

\renewcommand{\ker}[1]{\ensuremath{\mathrm{Ker} ({#1}) }}

\newcommand{\re}[1]{\ensuremath{\mathrm{Re} \left( {#1} \right) }}

\numberwithin{equation}{section}
\numberwithin{subsection}{section}

\newtheorem{thm}[subsection]{Theorem}

\newtheorem{lemma}[subsection]{Lemma}
\newtheorem{prop}[subsection]{Proposition}
\newtheorem{cor}[subsection]{Corollary}
\newtheorem*{thm*}{Theorem}

\theoremstyle{definition}
\newtheorem{defn}[subsection]{Definition}
\newtheorem{remark}[subsection]{Remark}
\newtheorem{eg}[subsection]{Example}
\newtheorem{quest}{Question}

\title[NC Blaschke-Singular-Outer Factorization]{Blaschke-Singular-Outer Factorization of Free Non-commutative Functions}

\author{Michael T. Jury}
\address{\scriptsize University of Florida}
\email{\scriptsize mjury@ad.ufl.edu}

\author{Robert T.W. Martin}
\address{\scriptsize University of Manitoba}
\email{\scriptsize Robert.Martin@umanitoba.ca}

\author{Eli Shamovich}
\address{\scriptsize Ben-Gurion University of the Negev}
\email{\scriptsize shamovic@bgu.ac.il }

\thanks{First named author partially supported by NSF grant DMS-1900364.}

\newcommand{\hardy}{\ensuremath{H^2(\B^d_{\N})}}
\newcommand{\mult}{\ensuremath{H^{\infty}(\B^d_{\N})}}

\begin{document}
\small

\begin{abstract}
By classical results of Herglotz and F.~Riesz, any bounded analytic function in the complex unit disk has a unique inner-outer factorization.  Here, a bounded analytic function is called \emph{inner} or \emph{outer} if multiplication by this function defines an isometry or has dense range, respectively, as a linear operator on the Hardy Space, $H^2$, of analytic functions in the complex unit disk with square-summable Taylor series. This factorization can be further refined; any inner function $\theta$ decomposes uniquely as the product of a \emph{Blaschke inner} function and a \emph{singular inner} function, where the Blaschke inner contains all the vanishing information of $\theta$, and the singular inner factor has no zeroes in the unit disk.

We prove an exact analogue of this factorization in the context of the full Fock space, identified as the \emph{Non-commutative Hardy Space} of analytic functions defined in a certain multi-variable non-commutative open unit ball. 
\end{abstract}

\maketitle
\onehalfspace

\section{Introduction}

Fundamental structure results of Herglotz and Riesz (and later Beurling) \cite{Herg11,Riesz22,Beu48} in the theory of analytic functions in the complex unit disk, $\D$, imply that any uniformly bounded analytic function, $h$, in $\D$ admits a \emph{Blaschke-Singular-Outer factorization}:
$$ h = \underbrace{b}_{\mbox{ \tiny Blaschke}} \cdot \underbrace{s}_{\mbox{\tiny Singular}} \cdot \underbrace{f}_{\mbox{\tiny Outer}},$$ where $b$ is an inner \emph{Blaschke product}, $s$ is a \emph{singular inner} and $f$ is an \emph{outer function}. There are several equivalent definitions of inner and outer functions in the unit disk. We will take operator-theoretic definitions as our starting point as these will most readily generalize to the non-commutative (NC) multi-variable setting of the full Fock space over $\C ^d$.

The Hardy space, $H^2 (\D )$, is the Hilbert space of analytic functions in the disk with square-summable Taylor series coefficients at the origin, and $H^\infty (\D )$ is the unital Banach algebra of all uniformly bounded analytic functions in $\D$. The Hardy algebra, $H^\infty = H^\infty (\D )$ can be identified with the \emph{multiplier algebra} of $H^2$, the algebra of all functions in $\D$ which multiply $H^2$ into itself. That is, if $f \in H^\infty$ and $g \in H^2$, then $f \cdot g = h \in H^2$, and multiplication by $f$ defines a bounded \emph{multiplier}, a bounded linear multiplication operator, $M_f$, on $H^2$.  One can then define $f \in H^\infty$ to be \emph{inner} if the multiplier $M_f$ is an isometry, or \emph{outer} if $M_f$ has dense range.  In particular, multiplication by the independent variable, $z$, defines an isometry on $H^2$, the \emph{shift}, $S = M_z$, so that $H^\infty = \mr{Alg}(I,S) ^{-weak-*}$ and this plays a central role in Hardy Space Theory \cite{Nik-shift,NF}. Blaschke and singular inner functions can also be described in purely operator-theoretic terms. Namely, given any $h \in H^\infty$ we define the shift-invariant space 
$$ \scr{S} (h) := \left\{ f \in H^2 \left| \ \frac{f}{h} \in \mr{Hol} (\D ) \right. \right\}, $$ of all $H^2$ functions `divisible by $h$'. Clearly $g \in \scr{S} (h)$ if and only if any zero of $h$ is a zero of $g$ with greater or equal multiplicity, and $\scr{S} (h) \supseteq \ran{M_h}$. An inner function, $\theta \in H^\infty$, is then a \emph{Blaschke inner} or \emph{singular inner} if
$$ \scr{S} (\theta ) = \theta H^2, \quad \mbox{or} \quad  \scr{S} (\theta ) = H^2, $$ respectively. Equivalently, $\theta$ is singular inner if it has no zeroes in the disk.  These are not the usual starting or historical definitions of Blaschke and singular inner functions, but they are equivalent, see \cite[Chapter 5]{Hoff} or \cite[Chapter III.1]{NF}.  The goal of this paper is to extend the seminal Blaschke-Singular-Outer factorization of functions in $H^\infty$ and $H^2$ to elements of the NC Hardy spaces.  


Recent research has identified the full Fock space over $\C ^d$, \be F^2_d  := \bigoplus _{k=0} ^\infty \left( \C ^d \right) ^{\otimes k}  = \C \oplus \C ^d \oplus \left( \C ^d \otimes \C ^d \right) \oplus \left( \C ^d \otimes \C ^d \otimes \C^d \right) \oplus \cdots, \label{Fockform} \ee
with the \emph{Free} or \emph{Non-commutative Hardy space}, $\hardy$, a canonical NC multi-variable analogue of $H^2 (\D )$ \cite{Pop-freeholo,Pop-freeholo2,Pop-freeharm,DP-inv,BMV,JM-freeCE,JM-NCFatou}. Elements of $H^2 (\B ^d _\N )$ are analytic matrix-valued functions defined in an NC multi-variable open unit ball, $\B ^d _\N$, in several NC matrix-variables \cite{Taylor,KVV,Ag-Mc,Voic,Voic2}: 
\be \B ^d _\N := \bigsqcup _{n=1} ^\infty \B ^d _n; \quad \quad  \B ^d _n := \left( \C ^{n\times n} \otimes \C ^{1\times d} \right) _1. \label{NCdisk} \ee Here, we fix the row operator space structure in $\B ^d _n$. Namely, any $d-$tuple of $n\times n$ matrices, $Z = (Z_1 , \cdots , Z_d ) \in \B ^d _n$, can be viewed as a linear map from $d$ copies of $\C ^n$ into one copy. The NC unit ball consists of the strict \emph{row contractions}, i.e, the $d-$tuples satisfying
$$ Z Z ^* = Z_1 Z_1 ^* + \cdots + Z _d Z_d ^* < I. $$ 

Elements of the full Fock space can be identified with power series in $d$ non-commuting variables with square-summable coefficients (see Section \ref{sec:prelim}). That is, any $f \in F^2_d$ is a power series:
$$ f(\mf{z} ) := \sum _{\alpha \in \F ^d } \hat{f} _\alpha \mf{z} ^\alpha, $$ where $\F ^d$, \emph{the free monoid on $d$ generators}, is the set of all words in the $d$ letters $\{ 1 , ... , d \}$, and given any word $\alpha = i_1 \cdots i_n,$ $i_k \in \{ 1 , ... , d \}$, $\mf{z} ^\alpha := \mf{z} _{i_1} \cdots \mf{z} _{i_n}$. At first sight this may appear to have little bearing to classical Hardy Space Theory and analytic function theory in the disk. However, foundational work of Popescu has shown that if $Z := (Z_1 , \cdots  , Z_d ) : \cH \otimes \C ^d \rightarrow \cH$ is any strict row contraction on a Hilbert space, $\cH$, then the above formal power series for $f$ converges absolutely in operator norm when evaluated at $Z$ (and uniformly on compacta) \cite{Pop-freeholo,SSS}. It follows that any $f \in F^2 _d$ can be viewed as a locally bounded \emph{free non-commutative function} in the NC open unit ball, $\B ^d _\N$ \cite{KVV}. That is, we can view $F^2 _d$ as the NC Hardy space, $\hardy$, the Hilbert space of all (analytic) free NC functions in $\B ^d _\N$ with square-summable Taylor series coefficients. Non-commutative $H^\infty$, $H^\infty (\B ^d _\N )$ can then be defined as the unital Banach algebra of uniformly bounded free NC functions in the NC open unit ball, and as in the single-variable setting, this can be identified (completely isometrically \cite{Pop-freeholo,SSS}) with the \emph{left multiplier algebra} of $H^2 (\B ^d _\N )$, the algebra of all free NC functions in $\B ^d _\N$ which left multiply the NC Hardy space, $H^2 (\B ^d _\N )$ into itself. Furthermore, again in exact analogy with classical Hardy Space Theory, left or right multiplication by the independent NC variables define isometries on the NC Hardy space: 
$$ L_k := M^L _{Z_k}, \quad R_k := M^R _{Z_k}, \quad \quad 1 \leq k \leq d, $$ and these have pairwise orthogonal ranges $L_k ^* L_j = I_{H^2} \delta _{k,j}$, so that the row operator: $L := \left( L_1 , L_2 , \cdots , L_d \right) : H^2 (\B ^d _\N ) \otimes \C ^d \rightarrow H^2 (\B ^d _\N )$ is an isometry which we call the \emph{left free shift}. The NC Hardy algebra, $H^\infty (\B ^d _\N)$ is equal to $\mr{Alg} (I,L) ^{-weak-*}$, the \emph{left free analytic Toeplitz algebra}. This algebra and its norm closed analogue were first studied by Popescu in \cite{Pop-vN} (see also \cite{Pop-multi}). Later they were also studied by Davidson and Pitts \cite{DP-inv,DP-alg, DP-pick,DLP-ncld}, Arias and Popescu \cite{AriasPopescu}, and further by Popescu \cite{Pop-dil,Pop-freeholo,Pop-freeholo2,Pop-freeharm}. In greater generality this setup was extensively studied by Muhly and Solel \cite{MS04,MS11,MS13}. 

Popescu was the first to discover an NC analogue of the classical Beurling theorem for $\hardy$ in \cite[Theorem 2.2]{Pop-charfun} (see also \cite[Theorem 4.2]{Pop-multfact} for the first instance of the inner-outer factorization). The theorem is also proven in \cite[Theorem 2.1]{AriasPopescu} and was later proven independently by Davidson and Pitts \cite[Corollary 2.2]{DP-inv}. Inner-outer factorization of NC functions in $\hardy$ or $\mult$ is an easy consequence of this; any $H \in \mult$ can be factored as $H = \Theta \cdot F$, where $\Theta$ is an NC inner (an isometric left multiplier) and $F$ is an NC outer, \emph{i.e.} $M^L_F = F(L)$ has dense range. Equivalently $F = M^L_F 1$ is an $R-$cyclic vector, and this second definition extends to $F \in \hardy$. In this paper, we refine these results to include an exact NC analogue of the Blaschke-Singular-Outer factorization. An NC Blaschke inner $B \in \mult$ will be an NC inner whose range is completely determined by its left `NC variety' in the NC unit ball. An NC inner left multiplier $S$ will be singular if  $S (Z)$ is invertible for any $Z \in \B^d_{\N}$.


\begin{thm*}[NC Blaschke-Singular-Outer factorization, Theorem \ref{NCBSOthm}]
Every non-zero $H \in H^{p}(\B^d_{\N})$, $p \in \{ 2 , \infty \}$, can be factored as a product $H = B \cdot S \cdot F$ for $B,S \in \mult$, where $B$ is an NC Blaschke inner with the same NC variety as $H$, $S$ is an NC singular inner and $F \in H^p (\B ^d _\N )$ is an NC outer function. The factors are unique up to scalars of unit modulus.
\end{thm*}
The left NC variety of any NC Hardy space function is formally defined in Definition \ref{leftsingdef} below. Roughly speaking, the NC variety is the collection of directional zeroes in the sense of \cite{HelMcC04} and \cite{HMP07}. When $d=1$, our NC Blaschke-Singular-Outer factorization theorem recovers the classical factorization with a new operator-theoretic proof, see Corollary \ref{dequals1}.

\subsection{Outline} Section \ref{sec:prelim} contains the necessary background on the NC unit ball, the NC Hardy space, the NC Hardy algebra $\mult$, and its commutant --- the algebra of right multipliers. In Section \ref{sec:nc_var} we discuss the (left) NC varieties cut out as degeneracy loci of functions in the NC Hardy spaces.  Examples of computations of NC Blaschke inner and singular inner functions are provided in Section \ref{sec:examples}. The main theorem stated above is proven in Section \ref{wanddim}. Lastly, the appendix contains a factorization result for NC idempotent-valued functions obtained while working on the main theorem and is of independent interest in our opinion.

\section{Preliminaries: Fock Space as the NC Hardy space} \label{sec:prelim}

The free monoid, $\F ^d$ is the set of all words in $d$ letters $\{ 1, ... , d \}$. This is the universal monoid on $d$ generators, with product given by concatenation of words, and unit $\emptyset$, the empty word containing no letters. The Hilbert space of square summable sequences indexed by $\F ^d$, $\ell ^2 (\F ^d )$, and $F^2 _d$, the direct sum of all tensor powers of $\C ^d$, \emph{i.e.} full Fock space over $\C ^d$, are naturally isomorphic (see equation \ref{Fockform}).  This isomorphism is implemented by the unitary map $e _{i_1 \cdots i_k} \mapsto e_{i_1} \otimes \cdots \otimes e_{i_k}$, $i_k \in \{ 1, ... , d \}$, and $e_\emptyset \mapsto 1$ where $\{ e_{j} \}$ denotes the standard basis of $\C ^d$, and $1$ is the vacuum vector of the Fock space (which spans the subspace $\C \subset F^2 _d$). Under this isomorphism the left free shifts become the left creation operators on the Fock space which act by tensoring on the left with the standard basis vectors of $\C ^d$. In the sequel we identify the free square-summable sequences, $\ell ^2 (\F ^d )$ and the Fock space $F^2 _d$ with the NC Hardy space, denoted by $H^2 (\B ^d _\N )$:
$$ H^2 (\B ^d _\N ) = \left\{  f \in \mr{Hol} (\B ^d _\N )  \left| \ f(Z) = \sum _{\alpha \in \F ^d } \hat{f} _\alpha Z^\alpha, \ \sum |\hat{f} _\alpha | ^2 < \infty \right.  \right\}. $$ Similarly, we will use the notation $H^\infty (\B ^d _\N ) := \mr{Alg} (I, L) ^{-weak-*}$,
$$ H^\infty (\B ^d _\N ) =   \left\{ f \in \mr{Hol} (\B ^d _\N )  \left| \ \sup _{Z \in \B ^d _\N } \| f (Z ) \| < \infty \right. \right\}.$$ Any element $F \in H^\infty (\B ^d _\N )$ is identified with the linear operator, $F(L) := M^L _{F}$, of left multiplication by $F(Z)$. As described in the introduction, $\mult$ can be identified with the left multiplier algebra of $\hardy$, and it immediately follows that $\mult \subset \hardy$. Any $f \in \hardy$ is a locally bounded free non-commutative function in the sense of modern Non-commutative Function Theory \cite{Taylor2,KVV,Ag-Mc}. That is, $f$ respects the grading, direct sums and similarities which preserve its NC domain, $\B ^d _\N$.  Any locally bounded free NC function (under mild, minimal assumptions on its NC domain) is automatically holomorphic, \emph{i.e.} it is both G\^{a}teaux and Fr\'{e}chet differentiable at any point $Z \in \B ^d _\N$ and has a convergent Taylor-type power series expansion about any point \cite[Chapter 7]{KVV}. 

The \emph{right free shifts}, $R_k = M^R _{Z_k}$ are unitarily equivalent to the left free shifts $L_k = M^L _{Z_k}$ via the transpose unitary on $\ell ^2 (\F ^d )$, $U_\dag$,
$$ U_\dag e_\alpha := e_{\alpha ^\dag}, $$ where if $\alpha = i_1 \cdots i_n \in \F ^d$, then
$ \alpha ^\dag := i_n \cdots i_1, $ its transpose. 

\subsection{Fock space as an NC reproducing kernel Hilbert space}
The Hardy space, $H^2 (\D) $ can be equivalently defined using Reproducing Kernel Theory. Namely, $H^2$ is the reproducing kernel Hilbert space (RKHS) of the \emph{Szeg\"{o} kernel}:
$$ k(z,w) := \frac{1}{1-zw^*}. $$ As in the single-variable setting, the Free Hardy Space $\hardy$ can be equivalently defined using (non-commutative) reproducing kernel theory \cite{BMV}. All non-commutative reproducing kernel Hilbert spaces (NC-RKHS) in this paper will be Hilbert spaces of free NC functions in the NC unit ball, $\B ^d _\N$. Any Hilbert space, $\cH$ of NC functions in $\B ^d _\N$, is a NC-RKHS if the linear point evaluation map, $K_Z ^* : \cH  \rightarrow \left( \C ^{n\times n}, \mr{tr} _n \right)$ is bounded for any $Z \in \B ^d _n$. We will let $K_Z$, the \emph{NC kernel map}, denote the Hilbert space adjoint of $K_Z ^*$, and, for any $y, v \in \C ^n$, 
$$ K \{ Z , y , v \} := K_Z (yv^*) \in \cH. $$ Furthermore, given $Z \in \B ^d _n, y, v \in \C ^n$ and $W \in \B ^d _m,  x, u \in \C ^m$ the linear map
$$ K(Z,W) [ \cdot ] : \C ^{n\times m } \rightarrow \C ^{n\times m}, $$ defined by
$$ \ipcn{y}{K(Z,W)[vu^*]x} := \ip{K \{Z , y, v \} }{ K \{ W , x, u \} } _{\cH}, $$ is completely bounded for any fixed $Z,W$ and completely positive if $Z=W$.  This map is called the completely positive non-commutative (CPNC) kernel of $\cH$. As in the classical theory there is a bijection between CPNC kernel functions on a given NC set and NC-RKHS on that set \cite[Theorem 3.1]{BMV}, and if $K$ is a given CPNC kernel on an NC set, we will use the notation $\cH _{nc} (K)$ for the corresponding NC-RKHS of NC functions.  The NC Hardy space, $\hardy$, is then the non-commutative reproducing kernel Hilbert space (NC-RKHS) corresponding to the CPNC Szeg\"{o} kernel on the NC unit ball, $\B ^d _\N$:
$$ K(Z,W) [ \cdot ] := \sum _{\alpha \in \F ^d} Z^\alpha [ \cdot ]  (W^\alpha) ^*; \quad \hardy = \cH _{nc} (K). $$ Adjoints of left multipliers have a familiar and natural action on NC kernel vectors:
\be F(L) ^* K \{ Z , y , v \} = K \{ Z , F(Z) ^* y , v \}. \label{NCadjmult} \ee

For our purposes, it will be convenient, as in \cite{Pop-freeholo}, to view elements of the NC Hardy spaces as holomorphic (locally bounded) NC functions on all strict row contractions on a separable Hilbert space. That is, we will add the infinite level to $\B ^d _\N$:
\be \B ^d _{\aleph _0} := \B ^d _\N \bigsqcup \B ^d _\infty, \ee where 
$$ \B ^d _\infty := \left( \C ^{\infty \times \infty } \otimes \C ^{1\times d} \right) _1, $$ denotes the set of all strict row contractions on the separable Hilbert space $\C ^\infty := \ell ^2 (\N )$, and $\C ^{\infty \times \infty } := \mc{L} (\ell ^2 (\N )$. Here, and throughout, the notation $\C ^{n \times m}$ denotes the $n\times m$ matrices with entries in $\C$, so that $\C ^{1\times d}$ is a row with $d$ entries. We will write $\C ^d$ in place of $\C ^{d \times 1}$.

\section{NC Varieties} \label{sec:nc_var}

Let $H(Z)$ be any free NC function in one of the NC Hardy spaces $\hardy$ or $\mult$. The left NC variety of $H$ is the appropriate analogue of a variety in our NC multi-matrix-variable setting. The definition below is stated more generally for operator-valued left multipliers between vector-valued NC Hardy spaces. Let $\mc{H}, \J$ be separable or finite-dimensional Hilbert spaces. We will write $\mult \otimes \mc{L} (\cH, \J )$ in place of the weak operator topology (WOT) closure of this algebraic tensor product, viewed as left multiplication operators from $\hardy \otimes \cH$ into $\hardy \otimes \J$.

\begin{remark} \label{matrixNC}
Any element $F (L) \in \mult \otimes \C ^{n \times m}$ or $\mult \otimes \mc{L} (\J , \cH )$ can be viewed as a matrix- or operator-valued function whose entries are bounded, free non-commutative functions in $\B ^d _\N$ or $\B ^d _{\aleph _0}$. Note, however, that $F (Z)$, viewed as a function in $\B ^d _\N$ need not be NC in the sense that it will generally not preserve direct sums. It can, however, be identified with a matrix-valued NC function, $\wt{F} (Z)$ (\emph{i.e.} $\wt{F}$ does preserve direct sums, joint similarities and the grading) defined by conjugating $F(Z)$ with appropriate basis permutation matrices \cite[pp. 65--66]{KVV-rational2}, \cite[p.38]{PV-realize}.  
\end{remark}

\begin{defn} \label{leftsingdef}
Given any $H \in \mult \otimes \mc{L} (\mc{H}, \J) $ or $H \in \hardy \otimes \cH$, the \emph{left singularity locus} or \emph{left NC variety} of $H$ is:
\begin{align*} \mr{Sing}  (H ) & := \bigsqcup _{n \in \N \cup \{ \infty \} } \mr{Sing} _n (H )  \\ \mr{Sing} _n (H ) & :=\left\{ \left. (Z, y ) \in \B ^d _n \times \C ^n \  \right| \ y^* H (Z) \equiv 0  \right\}. \end{align*} 

The (left) \emph{singularity space} of $H$ is:
$$ \scr{S} ( H ) := \{ h \in \hardy \otimes \J \ | \ y^* h(Z) \equiv 0 \ \forall \ (Z,y) \in \mr{Sing} (H) \}. $$
\end{defn}

The singularity space of any such $H$ (in vector-valued NC $H^2$ or operator-valued NC $H^\infty$) is clearly right shift invariant, and 
$$ \scr{S} (H) \supseteq \ran{H(L)}. $$ In the above $y^* H(Z) \equiv 0$ for $H (L) \in \mult \otimes \mc{L} (\cH, \J )$ and $Z \in \B ^d _n$, $y \in \C ^n$ means that
$$ \ip{y \otimes g}{H(Z) x \otimes h }_{\C ^n \otimes \J} =0, $$ for any $h \in \cH, g \in \J$, and any $x \in \C^n$.
\begin{remark}
Note that these varieties differ from the ones considered in \cite{Ami57,SSS,SSS2} since these varieties correspond to a left ideal in the algebra of right multipliers and not to two-sided ideals. Similar varieties in the case of NC polynomials and NC rational functions were considered by Helton and McCullough \cite{HelMcC04} and Helton, Klep and Putinar \cite{HMP07}. The projection onto the first coordinate gives the variety of determinental zeroes considered, for example, in \cite{HKM18}.
\end{remark}

\begin{remark}
Let $H \in H^p (\B ^d _\N)$, $p\in \{ 2, \infty \}$, and let $\pi \colon \bigsqcup_{n \in \N } \B^d_n \times \C^n \to \B^d_{\N }$ be the projection onto the first coordinate. We claim that if $\pi (\mr{Sing}(H)) = \B^d_{\N }$, then $H\equiv 0$. In other words, if $H$ is not identically zero, then one cannot have $\mr{det} H(Z) =0$ for all $Z \in \B ^d _\N$. Indeed, by \cite[Theorem 5.7]{KVV-germs} the inner rank of $H$ considered as a $1\times 1$ matrix over the ring of germs of uniformly analytic NC functions at $0$ is given by $\max_n \left\{ \left. \frac{\rank{H(Z)}}{n} \right| Z \in  \text{a neighbourhood of 0}  \cap \B ^d _n \right\}$. This latter number is less than $1$ since $\det H(Z) = 0$ for every $Z \in \B^d_{\N}$. Since the inner rank of $H$ is either $1$ or $0$ we conclude that the inner rank of $H$ is $0$. However, this can only happen, if $H \equiv 0$.
\end{remark}

\begin{defn}
An NC left multiplier, $H (L) \in \mult \otimes \mc{L} (\cH , \J )$, is:
\bn
\item \emph{inner}, if $H (L )$ is an isometry. 
\item \emph{outer}, if $H(L)$ has dense range in $\hardy\otimes \J$.
\en
An element of Fock space, $h \in \hardy$, is called \emph{NC outer} if it is cyclic for the right shifts.
\end{defn}
The second definition of an NC outer $h \in \hardy$ is equivalent to the first if $H \in \mult$. That is, if $H (L) \in \mult$, then $h:= H(L) 1 \in \hardy$ is NC outer if and only if $H$ is NC outer. (In fact, any element $h \in \hardy$ can be identified with a closed, densely-defined and generally unbounded left multiplier, $h(L)$ in the NC Smirnov class \cite{JM-freeSmirnov}. Under this identification, $h \in \hardy$ is NC outer if and only if $h(L)$ has dense range.)
\begin{defn}
An NC inner (isometric) left multiplier $\Theta \in \mult \otimes \mc{L} (\mc{H}, \mc{J})$ is:
\bn
    \item \emph{Blaschke} if $\ran{\Theta (L)} = \scr{S} (\Theta )$.

	\item \emph{singular}  if $\scr{S}(\Theta) = \hardy \otimes \mc{J}$.
\en
\end{defn}
\begin{remark}
A scalar NC inner $S \in \mult$ is singular if and only if it is pointwise invertible in the NC unit ball, $\B ^d _{\aleph _0}$. Indeed, since the constant functions are in $\scr{S}(S)$, the singularity locus of $S$ is empty. Thus for every $0 < r < 1$, the operator $S(rL)$ has dense range, i.e, it is an outer. By Theorem \ref{InvOut} $S(rL)$ is invertible and thus $S(Z)$ is invertible for every $Z \in \B^d_{\aleph_0}$.
\end{remark}

For simplicity, the following results are stated for scalar-valued NC left multipliers. These extend naturally to operator-valued left multipliers between vector-valued NC Hardy spaces. 
\begin{prop} \label{singlocus}
Given any $H \in H^p (\B ^d _\N)$, $p \in \{ 2, \infty \}$, $\mr{Sing}(H)$ satisfies the following properties:
\begin{enumerate}

\item If $(Z,y)$, $(W,x) \in \mr{Sing}(H)$ and $c \in \C$, then $(Z \oplus W, y \oplus c\cdot x) \in \mr{Sing}(H)$.

\item For $S \in \mr{GL}_n$ and $(Z,y) \in \mr{Sing}(H)$, such that $S^{-1} Z S \in \B_n^d$, we have that $(S^{-1} Z S, (S^*) ^{-1} y ) \in \mr{Sing}(H)$.
\end{enumerate}
\end{prop}

\begin{lemma}
Given any $H \in \mult$ or $\hardy$, the set $\scr{S} (H )$ is a closed, $R-$invariant subspace and 
$$ \scr{S} (H ) ^\perp = \bigvee _{(Z,y) \in \mr{Sing} (H ) } K \{ Z, y , v \}. $$ 
\end{lemma}
\begin{proof}
Clearly this is a subspace. If $f \in \scr{S} (H )$ then for any $(Z,y) \in \mr{Sing} (H )$, we have that 
$$ y^* (R_k f) (Z)   =  y^* f(Z) Z_k  = 0, $$ so that $R_k f \in \scr{S} (H )$. 
Observe that $f \in \scr{S} (H )$ if and only if
\ba 0 & = & \ipcn{y}{f(Z) v} \nn \\
& = & \ip{K\{ Z , y ,v \}}{f}_{\hardy}, \nn \ea for all $(Z,y) \in \mr{Sing} (H )$ and all 
$v \in \C ^n$. Hence if $(f_n) \subset \scr{S} (H )$ and $f_n \rightarrow f$ in norm, then for any $(Z,y) \in \mr{Sing} (H )$ so that $Z \in \B ^d _n$, and for any $v \in \C ^n$, 
\ba \ipcn{y}{f(Z)v} & = & \ip{K \{ Z, y , v \} }{f}_{\hardy} \nn \\
& = & \lim _{n\rightarrow \infty } \ip{K \{ Z, y , v \} }{f_n }_{\hardy} \nn \\
& = & \lim \ipcn{y}{f_n (Z) v} = 0. \nn \ea This proves that $\scr{S} (H )$ is closed.
\end{proof}

\begin{lemma}
If $\Theta \in \mult$ is NC inner then the kernels of the NC-RKHS $\left( \Theta (L) \hardy\right) ^\perp$ have the form:
$$ K ^\Theta \{ Z , y ,v \} := K\{ Z , y ,v \} - \Theta (L) K \{ Z, \Theta (Z) ^* y , v \}. $$
\end{lemma}
\begin{proof}
Easy to verify since $I - \Theta (L) \Theta (L) ^*$ is the orthogonal projector onto $\left( \Theta (L) \hardy\right) ^\perp$.
\end{proof}

\begin{cor}
If $(Z,y) \in \B ^d _n \times \C ^n$ belongs to the singularity locus of an NC inner $\Theta (L)$, then
\begin{equation} \label{eq:dBr_is_Szego}
K^\Theta \{ Z , y , v \} = K \{ Z , y , v \}.
\end{equation}
Conversely, if $v$ is cyclic for $\mr{Alg} (I, Z)$ and \eqref{eq:dBr_is_Szego} holds, then $(Z,y)$ is in the singularity locus.
\end{cor}
\begin{proof}
Clearly, since $\Theta(L)$ is injective, we have that \eqref{eq:dBr_is_Szego} holds if and only if $K\{Z, \Theta(Z)^* y, v\} = 0$. The latter holds if and only if for every $f \in \hardy$ we have
\[
0 = \langle K\{Z,\Theta(Z)^* y,v\}, f \rangle = \langle \Theta(Z)^* y, f(Z) v \rangle.
\]
Hence, if $(Z,y)$ is in the singularity locus, then the above equation holds. Conversely, if $v$ is cyclic, then the set of all $f(Z) v$ as $f$ ranges over $\hardy$ is a dense set and thus $(Z,y)$ is in the singularity locus.
\end{proof}

\begin{remark}
The above is not an if and only if statement in general. To see this consider $Z = \left( \begin{smallmatrix} A & B \\ 0 & C \end{smallmatrix} \right) \in \B ^d _n$ and set $v = \left( \begin{smallmatrix} v_0 \\ 0 \end{smallmatrix} \right)$, and $y = \left( \begin{smallmatrix} 0 \\ y_0 \end{smallmatrix} \right)$, for some $v_0,y_0 \neq 0$. Then for every $f \in \hardy$ we have $f(Z) v = \left( \begin{smallmatrix} f(A) v_0 \\ 0 \end{smallmatrix} \right)$ and thus
\[
\ip{K\{Z,y,v\}}{f}_{\hardy} = \ipcn{y}{f(Z)v} = 0.
\]
Also for every $f$, $f(Z)^* y = \left( \begin{smallmatrix} 0 \\ f(B)^* y_0 \end{smallmatrix} \right)$ and thus $K\{Z,f(Z)^*y,v\} = 0$ for every $f$. However, it need not be the case that $f(B)^* y_0 = 0$. This defect can be removed by relaxing our definition of NC variety: Let the \emph{extended NC variety} of $H \in \mult$ be the graded set:
$$  \mr{Sing} ' (H) := \bigsqcup _{n \in \N \cup \{ \infty \}} \mr{Sing} ' _n (H), $$ where
$$ \mr{Sing} ' (H) := \left\{  (Z,y,v) \left| \ Z \in \B ^d _n, \ y,v \in \C ^n; \ H(Z) ^* y \perp \mr{Alg} (I, Z) v \right. \right\}. $$ 
The \emph{extended singularity space} is then,
$$ \scr{S} ' (H) := \{ h \in \hardy | \ h(Z) ^* y \perp \mr{Alg} (I , Z) v \ \forall \ (Z,y,v) \in \mr{Sing} ' (H) \}. $$ It is easily verified that this space is again $R-$invariant, closed, and that 
$$ \scr{S} ' (H) ^\perp = \bigvee _{(Z,y,v) \in \mr{Sing} ' (H)} K \{ Z, y ,v \}. $$ 
Moreover, with this definition, $( Z , y , v ) \in \mr{Sing} ' (H)$ if and only if $K\{ Z , y ,v \} \in \scr{S} ' (H) ^\perp$.  Our original definition is, however, fully justified by the NC Blaschke-Singular-Outer factorization theorem. 
\end{remark}

\begin{lemma} \label{Blakernel}
An NC inner $\Theta $ is Blaschke if and only if 
$$ \ran{\Theta (L)} ^\perp = \bigvee _{\substack{(Z,y) \in \mr{Sing} _n (\Theta ); \\
v \in \C ^n; \ n \in \N \cup \{ \infty \} }} K \{ Z , y , v \}. $$ 
\end{lemma}
\begin{proof}
First any such Szeg\"{o} kernel vector is in $\ran{\Theta (L) } ^\perp$ by the last corollary.
By definition, $\Theta$ is Blaschke if the range of $\Theta (L)$ is exactly the set of all $f \in \hardy$ so that 
$$ y^* f(Z) =0, \quad \forall \ (Z,y) \in \ \mr{Sing} (\Theta ). $$ 
and this condition holds if and only if 
$$ \ip{K\{ Z,y,v \} }{f} =0, $$ for all $(Z,y)$ in this singularity locus. This, in turn, is equivalent to the corresponding set of NC Szeg\"{o} kernels spanning the orthogonal complement of the range of $\Theta (L)$.
\end{proof}

\section{NC Blaschke row-column factorization} \label{sec:main}

By the NC inner-outer factorization theorem, any NC Hardy space function, $H \in H^p (\B ^d _\N )$, $p \in \{ 2 , \infty \}$, in the NC unit ball factors uniquely as $H (L) = \Theta (L) \cdot F (L)$, where $\Theta \in \mult$, $\Theta$ is NC inner and $F \in H^p (\B ^d _\N )$ is NC outer \cite[Theorem 4.2]{Pop-multfact}, \cite[Corollary 2.2]{DP-inv}, \cite[Theorem 2.1]{AriasPopescu}. (For the inner-outer factorization of operator-valued left multipliers between vector-valued NC Hardy spaces, see \cite[Theorem 1.7]{Pop-entropy}.) In this section, we therefore start with an NC inner function $\Theta \in \mult$ and decompose it as the product of an NC Blaschke inner left row multiplier and an NC inner left column multiplier.

\begin{prop} \label{rowcolprop}
Any NC inner $\Theta \in \mult$ factors as 
$$ \Theta := B \cdot S = \bsm B_1, & \cdots, & B_N \esm \bsm S_1 \\ \vdots \\ S_N \esm. $$ where $\ran{ B  (L) } = \scr{S} (\Theta )$, $\mr{Sing} (\Theta ) = \mr{Sing} (B)$, $B$ is an NC Blaschke inner, all components $B_k (L) $ are inner with pairwise orthogonal ranges, and the column $S$ is also inner. 
\end{prop}
\begin{proof}
By \cite[Theorem 2.1, Corollary 2.2]{DP-inv} or \cite[Theorem 1.7]{Pop-entropy}, there is a (row) inner $B  (L) : \hardy\otimes \C ^N \rightarrow \hardy$ (where $N \in \N \cup \{ \infty \})$), so that the $R-$invariant subspace
$$ \scr{S} (\Theta ) = \ran{B  (L)}. $$ If $f=\Theta(L) g \in \ran{\Theta (L)}$, observe that for any $(Z,y) \in \mr{Sing}  (\Theta ) $, that 
$$ y^* f(Z) = y^* \Theta (Z) g(Z) = 0, $$ and it follows that $\ran{\Theta (L) } \subseteq \ran{B  (L)}$. Since both $B(L), \Theta (L)$ are isometries, this implies $\Theta (L) \Theta (L) ^* \leq B(L) B(L) ^*$ so that by the Douglas Factorization Lemma \cite{DFL}, there is a contraction, $S : \hardy \rightarrow \hardy \otimes \C ^N$ so that 
$$ \Theta (L) = B  (L) \cdot S, $$ and $\ran{S} \subseteq \ker{B(L)} ^\perp$. Moreover, 
$$ R_k \Theta (L) = B  (L) (R_k \otimes I_N) S = \Theta (L) R_k = B  (L) S R_k, $$ so that 
$$ B(L) ((R_k\otimes I_n) S - S R_k ) = 0, $$ and since $B(L)$ is an isometry
$$(R_k\otimes I_n) S - S R_k  =0. $$ The weak$-*$ closed unital algebra of the NC right shifts is the commutant of $\mult$ \cite[Theorem 1.2]{DP-inv}, and it follows that $S = S (L) \in \mult \otimes \C ^N$ is a column of left multipliers so that 
$$ \Theta (L) = B  (L) S (L) = \bsm B_1 (L), & \cdots , & B_N (L) \esm \bsm S_1 (L) \\ \vdots \\ S_N (L) \esm. $$ In the above, since $\Theta (L) , B  (L)$ are isometries, it follows that $S (L)$ is also an isometry (or inner), and also each $B_k (L)$ is an isometry, so that the $B_k (L)$ must have pairwise orthogonal ranges.
\end{proof}

Our goal is to show that $N=1$ so that both $B$ and $S $ are scalar NC inner functions, and it will further follow that $S$ is a scalar NC singular inner. 
\begin{thm} \label{InvOut}
    If $f \in \hardy$ is an NC outer, then $f(rL) \in \mult$ is invertible for $0 \leq r <1$.
\end{thm}
We will have several occasions to use the following concept of \emph{argument re-scaling map}:
\begin{defn}
Given any $r \in [0,1]$, let $\Phi _r : \hardy \otimes \cH \rightarrow \hardy \otimes \cH$ be defined by: 
\ba \Phi _r f & = & \Phi _r \sum _{\alpha \in \F ^d } L^\alpha 1 \otimes \hat{f} _\alpha \nn \\
& := & \sum _{\alpha } L^\alpha 1 \otimes r^{|\alpha | } \hat{f} _\alpha =: f_r. \nn \ea
Similarly define $\varphi _r : \mult \otimes \mc{L} (\cH , \J ) \rightarrow \mult \otimes \mc{L} (\cH , \J )$ by $\varphi _r F(L) = F(rL)$.
\end{defn}
We sometimes write $f_r = f(rL) 1$. If $F \in \mult$, then $\Phi _r F(L) 1 = \varphi _r ( F(L) ) 1$.
\begin{lemma}
For any $0<r \leq 1$, $\Phi _r$ is a contractive, self-adjoint quasi-affinity. The map $\varphi _r$ is a completely contractive homomorphism for any $r \in [0,1]$. If $\varphi _r : \mult \otimes \mc{L} (\cH ) \rightarrow \mult \otimes \mc{L} (\cH )$, then it is also unital, and extends to a completely positive and unital map on the corresponding operator system. The map $\Phi _r$ respects the module intertwining action of $\mult \otimes \mc{L} (\cH , \J )$: If $F(L) \in \mult \otimes \mc{L} (\cH , \J)$ and $f \in \hardy \otimes \cH$, then $\Phi _r F(L) f = F(rL) f_r$. 
\end{lemma}
\begin{lemma}
If $r \in [0,1)$, and $f \in \hardy$, then $f(rL) := M^L _{\Phi _r f} \in \mult$.
\end{lemma}
\begin{proof}
Write $f = \sum _{n=0} ^\infty f_n $, where each $f_n \in \fp$ is a homogeneous NC polynomial of degree $n$. (This is the Taylor-Taylor series expansion of $f$ at $0 \in \B ^d _1$.) Then $f_r = \sum r^n f_n$, and the operator norm of $f_r$ is 
\ba \| f _r (L) \|  & \leq & \sum _{n=0} ^\infty r^{n} \underbrace{\| f_n (L) \| _{\mc{L} (\hardy )}}_{=\| f_n (L) 1 \| _{\hardy}} \nn \\ & = & \sum _{n=0} ^\infty r ^n \| f_n \| _{H^2}, \nn \ea (the operator norm of any homogeneous free polynomial in $L$ coincides with its Fock space norm), 
\ba
& \leq & \sqrt{\frac{1}{1-r^2}} \cdot \left( \sum \| f_n \| _{H^2 } ^2 \right) ^{1/2} \nn \\
& = & \| f \|  _{H^2} \sqrt{\frac{1}{1-r^2}}. \nn \ea 
\end{proof}
\begin{proof}{ (of Theorem \ref{InvOut})}
Any NC outer $F \in \mult$ is necessarily pointwise invertible in the NC unit ball, $\B ^d _\N$ \cite[Lemma 3.2]{JM-freeSmirnov}, and this extends to any NC outer $f \in \hardy$. (Otherwise there is a $Z \in \B ^d _n$ and $y \in \C ^n$ so that $f(Z) ^* y =0$ and therefore $K\{Z , y , v \}$ is orthogonal to the $R-$cyclic subspace generated by $f$, for any $v \in \C ^n$.) By the previous lemma, $f(rL) \in \mult$ is uniformly bounded. If $f(rL)$ is not invertible, then it follows that $f(rZ) ^{-1}$ is not uniformly bounded in $ \B ^d _\N$, or, equivalently, $f(Z) ^{-1}$ is not uniformly bounded in $ r \B ^d _\N$. Since $\| f(Z) ^{-1} \| = \| ( f(Z) ^* ) ^{-1} \| $, $(f(Z) ^*) ^{-1} $ is not uniformly norm-bounded in $r \B ^d _\N$, and it follows that we can find a sequence $(W ^{(n)} ) \subset r \B ^d _\N$, $W^{(n)} \in r \B ^d _{m_n}$, and $y_n \in \C ^{m_n} $, $\| y _n \| =1$, so that
$$ \| f (W ^{(n)} ) ^*  y _n \| < \frac{1}{n}. $$ We view each level $\C ^n$ as a subspace of $\C ^\infty = \ell ^2 (\N )$ (the span of the first $n$ standard basis vectors) so that each $y_n \in \C ^\infty$. Let $\{ e_k \}$ be the standard orthonormal basis for $\C ^\infty$, and choose a unitary $U_n$ so that $U_n y_n = e_1$. Then, since $f(Z)$ is a free NC function,
\ba \| f (U_n W^{(n)} U_n ^* ) ^* e_1 \| & = & \| U_n f (W^{(n)} ) ^* U_n ^* e_1 \| \nn \\
& = & \| f (W^{(n)} ) ^* y_n \| \rightarrow 0. \nn \ea It follows that we can assume, without loss in generality, that $y_n = e_1$ for every $n \in \N$. That is, we can replace the uniformly bounded sequence of strict row contractions $W^{(n)}$, with the sequence $Z^{(n)} := U_n W^{(n)} U_n ^*$, and we set $y = e_1 = v$. Since $\| Z^{(n)} \| \leq r$ for every $n \in \N$, it follows that the sequence of NC Szeg\"{o} kernels $\left( K \{ Z^{(n)} , e_1 , e_1 \} \right)$ is uniformly bounded in Fock space norm:
\ba \| K \{ Z ^{(n)} , e_1 , e_1 \} \| _{\hardy} ^2 & = &  \left( e_1 , K ( Z^{(n)} , Z^{(n)} ) [E_{11} ] e_1 \right) _{\C ^\infty}  \nn \\
& \leq & \left( e_1 , K(Z^{(n)} , Z^{(n)} ) [ I ] e_1 \right) \nn \\
& \leq & \| K (Z^{(n)} , Z^{(n)} ) [ I ] \| \nn \ea
\ba
& = & \left\| \sum _{k=0} ^\infty \mr{Ad} _{Z^{(n)}, (Z^{(n)}) ^*} ^{(k)} (I) \right\| \nn \\
& \leq & \frac{1}{1-r^2}. \nn \ea In the above $\mr{Ad} _{Z, Z^*}$ denotes the completely positive map of adjunction by $Z$ and $Z^*$,
$$ \mr{Ad} _{Z, Z^*} (P ) := Z_1 P Z_1 ^* + \cdots + Z_d P Z_d ^*, $$ and we used the fact that $Z Z^* \leq r^2 I$. Since this sequence of NC kernels is uniformly bounded, it follows that there is a weakly convergent subsequence, $\left( K \{ Z^{(k)} , e_1 , e_1 \} \right)$ (where say $k=n_k$) so that 
$$ K\{ Z^{(k)} , e_1 , e_1 \} \stackrel{w}{\rightarrow} h \in \hardy. $$ The vacuum coefficient of $h$ is:
\ba h_\emptyset & = & \ip{1}{h} _{\hardy} \nn \\
& = & \lim \ip{1}{K \{ Z^{(k)} , e_1 , e_1 \} }_{\hardy} \nn \\
& = & \left( e_1 , e_1 \right) _{\C ^\infty} \nn \\
& = & 1, \nn \ea and hence $h \neq 0$.  However, for any NC polynomial $p \in \fp$, consider:
\ba | \ip{h}{p(R)f}_{\hardy} |  & = &  \lim \left| \ip{K\{ Z^{(k)} , e_1 ,  e_1 \}}{ p(R) f} \right| \nn \\
&  =  & \lim  \left| \left( e_1 , f (Z ^{(k)} ) p (Z^{(k)}) e_1 \right) _{\C ^\infty } \right|  \nn \\
& \leq  & \lim \| f (Z^{(k)}) ^* e_1 \| \| p \| _{\mult}  \\
& = & 0, \nn \ea by assumption. Since $p \in \fp$ was arbitrary and $h \neq 0$, we conclude $f$ is not $R-$cyclic, contradicting the assumption that $f \in \hardy$ is NC outer.
\end{proof}

\begin{cor} \label{innersing}
Given any $H \in H^p (\B ^d _\N)$, $p \in \{ 2 , \infty \}$, if  $H= \Theta \cdot F$ is the inner-outer factorization of $H$, then $\mr{Sing} (H) = \mr{Sing} (\Theta)$.
\end{cor}
\begin{proof}
We have that $(Z,y) \in \mr{Sing} (H)$ if and only if 
$$ y^* H(Z) = y^* \Theta (Z) F(Z) =0. $$ Since $F$ is outer, it is pointwise invertible in $\B ^d _{\aleph _0}$ by the previous theorem, so that the above happens if and only if $(Z,y) \in \mr{Sing} (\Theta )$. 
\end{proof}

\begin{cor} \label{NCsingular}
For every $0<r<1$ there is an $H _r  (L ) \in \mult \otimes \C ^{1\times N}$  so that 
$$ B(rL) = \Theta (rL) H_r  (L), \quad \mbox{and} \quad  H_r (L) S(rL) = I _{\hardy}. $$ In particular, the NC column-inner, $S(L)$, is pointwise left invertible in the NC unit ball, $\B ^d _{\aleph _0}$.
\end{cor}
\begin{proof}
We have that $\Theta (L) = B  (L) S (L)$. For any $0<r<1$, let, 
$$ \Theta (rL) = \Ga _r (L) F_r (L), $$ be the inner-outer factorization of $\Theta (rL)$.
Fix $0<r<1$ and choose $0<s<1$ so that $s>r$. Then, if $0<t<1$ it follows that 
$$ \Theta (stL) = \Ga _s (tL) F_s (tL), $$ where now $F_s (tL)$ is an invertible left multiplier by Theorem \ref{InvOut} so that 
$$ \Ga _s (tL) =\Theta (st L) F_s (t L) ^{-1}. $$ 

By definition of $B  (L)$, it follows that if $(Z,y) \in \mr{Sing} (\Theta )$ so that 
$$ y^* \Theta (Z) =0, $$ then necessarily, 
$$ y^* B  (Z) =0, $$ and this shows that 
$$ \ran{\Theta (Z) } ^{\perp} \subseteq \ran{B  (Z) } ^\perp, $$ for any $Z \in \B ^d _n$, $n \in \N \cup \{ \infty \}$.  In particular, for any $0<r<1$, taking $Z = rL$,
$$ \ran{ B  (rL) } ^{- \| \cdot \| } \subseteq \ran{\Theta (rL) } ^{-\| \cdot \| } = \ran{\Ga _r (L) }. $$ Applying Douglas Factorization and using that $\mult$ is the commutant of the algebra of right multipliers \cite[Theorem 1.2]{DP-inv}, it again follows that there is a bounded left row multiplier $ G_r  (L)$ so that 
$$ B  (rL)  =  \Ga _r (L )  G_r  (L), $$ and finally,
\ba B  (st L) & = &  \Ga _s ( t L ) G _s  (t L) \nn \\
& = & \Theta (st L) \underbrace{F_s (t L) ^{-1} G _s   (t L) }_{=: \check{H}  _s ( t L) }. \nn \ea In particular, since we fixed $s>r$, we can choose $0<t<1$ so that $st = r$, and 
$$ B (r L) = \Theta (rL) H_r (L), \quad  \mbox{where} \quad  H_r (L) := \check{H} _{r/t} (tL) \in \mult \otimes \C ^{1\times N}.$$  This proves the existence of $H_r$. Since $B(rL) = \Theta (rL) H_r (L)$, and $\Theta (rL)$ is injective \cite[Theorem 1.7]{DP-inv}, it follows that $ \ker{B(rL)} = \ker{H_r (L)}$.

For any $0<r<1$,
\ba \Theta (rL) & = & B  (rL) S (rL) \nn \\ 
& = & \Theta (rL) H  _r (L) S (rL). \nn \ea Again, since $\Theta (rL)$ is injective, it follows that $ H  _r (L) S (rL) = I _{\hardy}$.
\end{proof}
\begin{cor} \label{idempotent}
The matrix-valued left multiplier
$$ E _r (L) := S (rL) H _r  (L) \in \mult \otimes \C ^{N \times N}, $$ is idempotent. For any $0<r<1$, 
$$ \ran{I-E_r(L)} = \ker{E_r (L)} = \ker{B(rL)} = \ker{ H_r (L)}. $$ 
For any $0<r, s<1$, $H_{r\cdot s} (L) = H_r (sL)$ and $E_{r\cdot s } (L) = E_r (sL )$.
\end{cor}
\begin{proof}
If we define $E_r (L) := S (rL) H  _r (L)$, then 
$$ E_r (L) E_r (L) = S (rL) \underbrace{H  _r (L) S (rL)}_{=I_{\hardy}} H  _r (L) = E_r (L), $$ proving that every $E_r$ is idempotent. Also,
\ba B (rL) & = & \Theta (rL)  H  _r (L) \nn \\
& = & B  (rL) E_r (L), \nn \ea and it follows that the idempotent 
$$ e_r (L) := (I_{\hardy} \otimes I_N) - E_r (L), $$ takes values in the kernel of $B  (rL)$. Conversely, consider $E_r (L) = S(rL) H_r (L)$. Clearly, $\ker{H_r (L)} \subseteq \ker{E_r (L)} = \ran{ I - E_r (L)} $, and on the other hand if $E_r (L) x =0$ then 
$$ 0 = H_r (L) E_r (L) x = \underbrace{H_r (L) S(rL)}_{=I} H_r (L) x, $$ so that $\ker{B(rL)} = \ker{H _r (L)} = \ker{E_r(L)} = \ran{I-E_r (L)}$.

Since $B(rL) = \Theta (rL) H_r (L)$, it follows that 
\ba \Theta (rs L) H_r (s L) & = & \varphi _s \left( \Theta (rL) H_r (L) \right) \nn \\
& = & \varphi _s (B_r (L) ) = B (rs L) \nn \\
& = & \Theta (rs L ) H_{rs} (L). \nn \ea 
It follows that 
$$ \Theta (rs L) \left( H_{rs} (L) - H_r (sL) \right) =0, $$ and since $\Theta (rs L)$ is injective, $H_r (sL) = H_{rs} (L)$. Then, by definition of $E_r (L)$, 
$$ E_{rs} (L) = S(rs L) H_{rs} (L) = S(rs L ) H_r (sL) = E_r (sL). $$ 
\end{proof}
\begin{remark}
By \cite[Corollary 1.8]{DP-inv}, the algebra $	\mult$ contains no non-trivial idempotents. This result can be extended in a natural way to $\mult \otimes \C ^{N \times N}$ to show that any NC idempotent $E \in \mult \otimes \C ^{N \times N}$  factors as:
$$ E  (L) = T (L) ^{-1} \left( I \otimes P \right) T (L), $$ where $T (L) \in \mult \otimes \C ^{N\times N}$ is invertible and $P \in \C ^{N \times N}$ is a fixed projection, see Appendix \ref{appendix}.
\end{remark}
\begin{remark} \label{rdepend}
Define operator-valued functions in $\B ^d _{\aleph _0}$ by
$$ H(Z) := H_r (Z/r), \quad \mbox{and} \quad E(Z) := E_r (Z/r ), $$ where if $\| Z \| = s <1$ then $r$ is any value so that $0<s<r<1$. This is well-defined since if $0<s=\| Z \| < r < t <1$, then 
$$ H(Z) = H_r (Z/r) = H _{t \cdot r/t} (Z/r) = H_t (Z/t). $$ Then $H, E$ can be identified with operator-valued free NC functions in $\B ^d _{\aleph _0}$ (see Remark \ref{matrixNC}), and they are uniformly bounded on balls $r \B ^d _{\aleph _0}$ of radius $r<1$.
\end{remark}

\section{NC Blaschke-Singular-Outer Factorization} \label{wanddim}

Consider the net of operator-valued left multipliers $B(rL) \in \mult \otimes \C ^{1\times N}$ for $0<r\leq 1$. Define the closed $R-$invariant subspaces $\scr{M} _r := \ran{B(rL)} ^{-\| \cdot \|}$, and let $Q_r$ denote the orthogonal projections onto these spaces. Recall then, that 
$$ P_r ^\perp := R (Q_r \otimes I_d ) R ^*, $$ is the projection onto the range of the row isometry $R | _{\scr{M} _r \otimes \C ^d}$, and that the \emph{wandering space} of $\scr{M} _r$ is defined to be the subspace:
$$ \scr{W} _r := \scr{M} _r \ominus R \scr{M} _r \otimes \C ^d, $$ with orthogonal projector
$$ P_r := Q_r - R (Q_r \otimes I_d ) R^*. $$ Elements of $\scr{W} _r = \ran{P_r}$ are called \emph{wandering vectors}, and if $\{ \Om _{r ; k } \} _{k=1} ^{N_r} $ is an orthonormal basis of wandering vectors then, 
$$ \Om _r (L) := \left( \Om _{r ; 1} (L) , \cdots , \Om _{r ; N_r } (L) \right); \quad \quad \Om _{r;k} (L) := M^L _{\Om _{r;k}}, $$ is a left-inner row multiplier with $\ran{\Om _r (L)} = \scr{M} _r$. We will call
$$ N_r  = \mr{dim} (\scr{W} _r ), $$ the \emph{wandering dimension} of $\scr{M} _r$. We then have the NC inner-outer factorization:
$$ B(rL) = \Om _r (L) F_r (L); \quad \Om _r (L) \in \mult \otimes \C ^{1 \times N_r}, \  F_r(L) \in \mult \otimes \C ^{N _r \times N}, $$ where $F_r (L) := \Om _r (L) ^* B(rL)$ is NC left outer for every $0<r \leq 1$. (Here, note that the Douglas Factorization Lemma implies the existence of a bounded linear operator $F_r$ so that $B(rL )= \Om _r (L) F_r$. Since $\Om _r (L)$ is an isometry, and $B(rL)$ is a contraction, $F_r ^* F_r = B(rL) ^* B(rL) < I$, so that $F_r$ is therefore also a contraction. Again using that $\Om _r (L)$ is an isometry, one can verify that each component of $F_r$ commutes with the right shifts, so that $F_r = F_r (L)$ is a left operator-valued multiplier \cite[Theorem 1.2]{DP-inv}. Moreover, $F_r =F_r(L)$ has dense range by construction, and is therefore NC outer.) The goal of this section is to prove that $B(rL)$ is injective for $0 < r \leq 1$. Our NC Blaschke-Singular-Outer factorization theorem will be an easy consequence of this fact.

The following lemma is a straightforward observation, but we would like to emphasize the distinction of the cases $d \in \N$ and $d=\infty$. In the notation of the previous discussion:
\begin{lemma} \label{lem:wand_proj_conv}
If $Q_r \stackrel{SOT}{\rightarrow}  Q$, then $P_r \stackrel{SOT}{\rightarrow}  P$.
\end{lemma}
\begin{proof}
This is a consequence of the fact that $Q_r \otimes I_d \stackrel{SOT}{\rightarrow}  Q \otimes I_d$. The convergence is immediate, if $d \in \N$. For $d=\infty$, this is equivalent to $Q_r \stackrel{\sigma-SOT}{\rightarrow} Q$. This latter claim follows from the fact that the $Q_r$ are bounded and \cite[Lemma 2.5]{Takesaki}. 
\end{proof}

The main part of the following lemma is implicit in the work of Davidson and Pitts \cite{DP-inv}.

\begin{lemma} \label{monrank}
Let $A (L) \in \mult \otimes \mc{L} (\cH, \J )$ be any left multiplier, and set $\scr{M} _r := \ran{A (rL)} ^{-\| \cdot \|} $. The wandering dimension of $\scr{M} _r$ is non-decreasing as $r \uparrow 1$. Furthermore, if $\scr{W}_r$ is the wandering subspace of $\scr{M}_r$ and $P_r$ is the projection onto $\scr{W}_r$, then $\scr{W}_r = \left( P_r \Phi_r \scr{W}_1\right)^{-\|\cdot\|}$.
\end{lemma}
\begin{proof}
It suffices to show that for every $0<r<1$, we have that $\scr{W}_r = P_r \Phi_r\left( \scr{W}_1\right)^{-\|\cdot\|}$. Indeed, this implies that $ \dim{\scr{W}_r } \leq \dim{\scr{W}_1}$. Moreover, for $0<t<r \leq 1$, set $C (L) = A(rL)$ and $s = t/r$, then $A(tL) = C(sL)$ and applying the lemma to $C(L)$ will yield $\dim{\scr{W} _t} \leq \dim{\scr{W} _r}$.

Now fix $0 < r < 1$, and let $\mc{W} _1 = \{w_1, \cdots ,w_k\}$ be an orthonormal basis of $\scr{W}_1$. Note that $\mc{W}_r = \Phi_r \mc{W} _1 \subset \scr{M}_r$. Moreover, note that $\Phi_r(\scr{M}_1)^{-\|\cdot\|} = \scr{M}_r$ since $\Phi_r(\hardy )$ is dense in $\hardy$. Since $r R_j \otimes I_\K \Phi_r = \Phi_r R_j \otimes I_\K$ for every $1 \leq j \leq d$ and NC polynomials in $R \otimes I_\mc{J}$ acting on $\mc{W} _1$ generate a dense linear subspace of $\scr{M}_1$, we conclude that $\mc{W}_r$ is $R\otimes I_\J-$cyclic in $\scr{M} _r$.
Let $P_r^{\perp}$ be, as above, the projection onto $\scr{M}_r \ominus \scr{W}_r$. Let $w \in \scr{W}_r \ominus P_r \mc{W}_r$ and $u \in \mc{W}_r = \Phi _r \mc{W} _1$ be arbitrary. Write $u = P_r u + P_r^{\perp} u$. For every multi-index $\alpha$ we obtain that
\[
\langle w, R^{\alpha} \otimes I_\J u \rangle = \langle w, R^{\alpha} \otimes I_\J P_r u \rangle = 0. 
\]
The first equality follows from the fact that $\scr{M}_r \ominus \scr{W}_r$ is $R-$invariant and the second since $w$ is wandering and orthogonal to $P_r \mc{W}_r$. Since $\mc{W}_r$ is a $R \otimes I_\J-$cyclic subset of $\scr{M} _r$, we conclude that $w \equiv 0$ so that $\scr{W} _r = \bigvee P_r \Phi _r \mc{W} _1 =  \left( P_r \Phi _r \scr{W} _1 \right)  ^{-\| \cdot \|}$.

\end{proof}

Let $T_r = I - Q_r$ be the projection onto $\ran{B(rL)} ^\perp$, for $r \in (0,1]$.

\begin{prop} \label{conranproj}
The projections $T_r \stackrel{SOT}{\rightarrow} T = I-Q = I - B(L) B(L) ^*$.
\end{prop}
\begin{lemma} \label{dilsing}
For any $0< r \leq 1$, $\mr{Sing} ( B(rL) )$ is the set of all $(Z,y)$ so that $(rZ, y) \in \mr{Sing} (\Theta )$. 
\end{lemma}
\begin{proof}
One has $y^* \Theta (rZ) = 0 $ if and only if $(rZ , y) \in \mr{Sing} (\Theta ).$ 
\end{proof}
\begin{lemma} \label{dilperp}
Let $\mc{S} \subset \hardy$ be any linear subspace. A vector $x \in \hardy$ is orthogonal to $\Phi _r \mc{S}$ if and only if $x_r = \Phi  _r x$ is orthogonal to $\mc{S}$.  
\end{lemma}
\begin{proof}
This follows immediately from the fact that $\Phi _r$ is self-adjoint.
\end{proof}
\begin{lemma} \label{kerform}
Any NC Szeg\"{o} kernel vector, $K \{ Z ,y , v \}$ for $Z \in \B ^d _n$, $y,v \in \C ^n$, $n \in \N \cup \{ \infty \}$ is given by the formula:
$$ K \{ Z , y ,v \} = \sum _{\alpha \in \F ^d} \ipcn{Z^\alpha v}{y} L^\alpha 1. $$ 
For any $r \in [0,1]$, $\Phi _r K \{Z ,y , v \} = K \{ r Z , y , v \}$.
\end{lemma}
\begin{proof}{ (of Proposition \ref{conranproj})}
We have that 
$$ \ran{T} = \ran{B(L)} ^\perp = \bigvee _{(Z,y) \in \mr{Sing} (B)} K\{ Z, y ,v \}. $$ 
Choosing a countable dense subset of kernel vectors and applying Gram-Schmidt orthgonalization (and using that linear combinations of NC kernels are NC kernels) we obtain an orthonormal basis 
$$ \{ K \{ Z(n), y _n , v_n \} \} _{n =1} ^\infty, $$ for $\ran{B(L)} ^\perp$. (Each $(Z(n), y_n)$ belongs to $\mr{Sing} (B)$, in fact, since linear combinations of NC kernels are NC kernels:
$$ K \{ Z , y , v \} + c K \{ W , x , u \} = K \{ Z\oplus W , y \oplus c \cdot x , v \oplus u \}, $$ 
and if $(Z,y), (W,x) \in \mr{Sing} (B)$, so is $(Z \oplus W , y \oplus c \cdot x)$ for any $c \in \C$, see Proposition \ref{singlocus}.  This is, however, not germane for our arguments here.)
Given any $N \in \N$, and any $0<r <1$, we define 
$T_r (N)$ as the orthogonal projection onto 
$$ \bigvee \left\{ \left. K \{ r  ^{-1} Z(n) , y _n , v _n \} \right| \ 1 \leq n \leq N \quad \mbox{and} \quad \| Z(n) \| < r \right\}. $$ Here, note that for any NC Szeg\"{o} kernel in the above set, $ \| Z (n ) \| / r <1$ so that each of these kernels is a well-defined vector in $\hardy$. If we choose $0<R_N <1$ so that 
$$ R_N = \mr{max} _{1\leq n \leq N} \| Z(n) \|, $$
then for any $r \in (R _N , 1 ]$, $T_r (N)$ is the projection onto 
$$ \bigvee _{1\leq n \leq N}  K \{ r  ^{-1} Z(n) , y _n , v _n \}. $$ 
We write $T(N) := T_1 (N)$.  Since $\Phi _r K\{ r ^{-1} Z(n), y _n , v _n \} = K\{ Z(n) , y_n , v _n \} \in \ran{B(L)} ^\perp$ by Lemma \ref{kerform}, each of the $K\{ r ^{-1} Z(n) , y_n , v_n \}$ belongs to $\ran{T_r} = \ran{B(rL)} ^\perp$ by Lemma \ref{dilperp}. (In fact, $(r^{-1} Z(n) , y_n ) \in \mr{Sing} (B(rL) )$ by Lemma \ref{dilsing}.) It follows that $T_r (N) \leq T_r$ for any $r \in ( 0 , 1]$. Further observe that by the formula of Lemma \ref{kerform},
$$ K\{ r^{-1} Z(n), y _n , v_n \} \stackrel{\longrightarrow}{\mbox{\tiny $r\uparrow 1$}} K \{ Z(n) , y_n , v _n \}, $$ 
so that $T_{r} (N) \stackrel{SOT}{\rightarrow} T(N)$ as $r \uparrow 1$. Moreover it is clear that $T(N) \stackrel{SOT}{\rightarrow} T$. 

Consider the net $(T_r )_{r \in (0, 1]}$. This is a net of projections, and hence is uniformly bounded for $0<r \leq 1$. Any subsequence $(T_{r_k} )$, for which $r_k \uparrow 1$ has a $WOT$ convergent subsequence.  Let $(T_{r_k} )$ be any such $WOT-$convergent subsequence so that as $r_k \uparrow 1$, $T_{r_k} \stackrel{WOT}{\rightarrow} \wt{T}$. Then, for any $N \in \N$,
\ba \ip{x}{\wt{T} x } & = & \lim \ip{x}{T_{r_k} x}  \nn \\
& \geq &  \lim \ip{x}{T_{r_k} (N) x} \nn \\
& = & \ip{x}{T (N) x}. \nn \ea Here, we note that since $r_k \uparrow 1$, we have that eventually $r_k > R_N$. This proves that $\wt{T} \geq T(N)$ for any $N \in \N$, and hence $\wt{T} \geq T$. Further note that $\wt{T}$ is positive semi-definite, and it is a contraction: Since $T_{r_k} \stackrel{WOT}{\rightarrow} \wt{T}$,
$$ \ip{x}{\wt{T} x}  =  \lim _k \ip{x}{T_{r_k} x} \geq 0. $$ Moreover,
\ba \| \wt{T} x \| ^2 & = & \lim | \ip{T_{r_k}x}{\wt{T} x} | \nn \\
&\leq & \limsup \| T _{r_k} x \| \| \wt{T} x \| \nn \\
& \leq & \|x \| \| \wt{T} x \|. \nn \ea This proves that $ \| \wt{T} x \| \leq \| x \|$, and $\| \wt{T} \| \leq 1$. Let $x = \wt{T} y$ be any vector in $\ran{\wt{T}}$. Then $x_k := T_{r_k} y \stackrel{w}{\rightarrow } x$, where $w$ denotes weak convergence. 
By Lemma \ref{dilperp}, we know that for each $k$, $ x_k \in \ran{B(r_k L)} ^\perp, $ so that 
$ h_k := \Phi _{r_k} x_k \in \ran{B(L)} ^\perp. $ 
Since each $\Phi _{r_k}$ is a contraction and so is $\wt{T}$, the sequence $h_k$ is uniformly bounded.  Then, for any $\alpha \in \F ^d$,
\ba \lim _k (h_k ) _\alpha \nn & = & \lim _k r_k ^{|\alpha |} (x_k )_\alpha \nn \\ 
& = &  x_\alpha, \nn \ea since $r_k \uparrow 1$, and $(x_k ) _\alpha \rightarrow x_\alpha$ since $x_k$ converges weakly to $x$. Since $\alpha \in \F ^d$ is arbitrary and the sequence $(h_k )$ is uniformly bounded, it follows that $h_k \stackrel{w}{\rightarrow} x$ (converges weakly to $x$). Moreover, each $h_k \in \ran{B(L)} ^\perp$, and closed subspaces are weakly closed, so that $x = wk-\lim h_k \in \ran{B(L)} ^\perp$ and we conclude that $\ran{\wt{T}} \subseteq \ran{T}$. Since $\wt{T}$ is a positive semi-definite contraction and $T$ is a projection, $T \wt{T} = \wt{T}$, and 
\ba T \wt{T} & = & \wt{T} = \wt{T} ^* \nn \\
& = & \wt{T} T = T \wt{T} T \nn \\
& \leq & T. \nn \ea This proves that $\wt{T} \leq T$. Earlier we proved that $\wt{T} \geq T$, and we conclude that $T = \wt{T} = WOT-\lim _k T_{r_k}.$ Since the subsequence $T_{r_k}$ was an arbitrary $WOT-$ convergent subsequence so that $r_k \uparrow 1$, it follows that the entire net $T_r$ converges in $WOT$ to $T$ as $r \uparrow 1$. Since each $T_r, T$ are projections, we then obtain that $T_r \rightarrow T$ in the strong operator topology. 
\end{proof}
\begin{remark}
Since $B(rL)$ converges $SOT-*$ to $B(L)$ as $r \uparrow 1$ (see \emph{e.g.} \cite[Lemma 6.3]{JM-NCFatou}), it follows that $B(rL) B(rL) ^* \stackrel{SOT}{\rightarrow} B(L) B(L) ^* = Q$ as $r \uparrow 1$. Since $Q$ is a non-trivial projection, its spectrum is $\{ 0 , 1 \}$, and it follows that for any $t \in (0,1)$, the spectral projections
$$ \chi _{[0, t]} ( B(rL) B(rL) ^*) \stackrel{SOT}{\rightarrow} (I -Q), $$ and 
$$ \chi _{[t,1]} (B(rL) B(rL) ^* ) \stackrel{SOT}{\rightarrow} Q, $$ where $\chi _{[a,b]}$ denotes the characteristic function of the interval $[a,b]$ \cite[Theorem VIII.24 (b)]{RnS1}. It does not immediately follow, however, that $Q_r = I - T_r $ converges to $Q$ because 
$$ Q_r = \chi _{(0, 1]} ( B(rL) ^* B(rL) ), $$ and $0$ belongs to the spectrum of $Q$, see \cite{RnS1}. The crucial fact that makes the above proof work is that if $B(L)$ is NC Blaschke, then $\ran{B(L)} ^\perp$ is spanned by NC functions which are each analytic in an NC ball of radius greater than $1$.
\end{remark}

\begin{cor} \label{Binjcor}
$B(rL)$ is injective for $r \in (0 ,1]$.
\end{cor}
\begin{lemma} \label{nzerovac}
If $0 \neq h \in \ker{B(rL)}$, there is an $h' \in \ker{B(rL)}$ so that $h'(0) \neq 0 \in \C ^N$. If $e =I-E$ is the NC idempotent so that $\ran{e(rL)} = \ker{B(rL)}$, then $e _\emptyset = e (0) \equiv 0$ vanishes identically if and only if $e \equiv 0$ is identically zero.
\end{lemma}
\begin{proof}
Observe that $\ker{B(rL)} = \ker{B(rL) ^* B(rL)}$. Indeed if $B(rL) h =0$ then $B(rL) ^* B(rL) h =0$. Conversely, if 
$B(rL) ^* B(rL) h =0$, then 
$$ 0 =  \ip{h}{B(rL) ^* B(rL) h} = \| B(rL) h \| ^2, $$ and it follows that $B(rL) h =0$. 

If $h_\emptyset = h(0) = 0$, Then $h = R \mbf{h} = R_1 h^{(1)} + \cdots +  R_d h ^{(d)}$ for some $\mbf{h} \in \hardy \otimes \C ^N \otimes \C ^d$. Then 
$$ 0 = R_k ^* B(rL) ^* B(rL) h = B(rL) ^* B(rL) h^{(k)}, $$ and it follows that $h^{(k)} \in \ker{B(rL)}$ for every $1\leq k \leq d$. If $h^{(k)} _\emptyset =0$, then we can repeat this process until we ultimately end up with a $g \in \ker{B(rL)}$ so that $g(0) \neq 0$. In more detail, if $\alpha \in \F ^d$ is any word of minimal length so that $h_\alpha \neq 0$, then $g := (R^\alpha ) ^* h \in \ker{B(rL)}$, and $g(0) = g _\emptyset = h _\alpha \neq 0$. 

If $e(0) \equiv 0$, then any $h \in \ker{B(rL)} =\ran{e(rL)}$ has the form $h = e (rL) g$ for some $g \in \hardy \otimes \C ^N$, so that $h(0) = e(0) g(0) = 0$. Hence there is no $h \in \ker{B(rL)}$ so that $h(0) \neq 0$. If there was a non-zero $h \in \ker{B(rL)}$, then by the above argument there would be a non-zero $g \in \ker{B(rL)}$ so that $g(0) \neq 0$. We conclude that $\ker{B(rL)} = \{ 0 \}$ and $e \equiv 0$.
\end{proof}
\begin{proof}{ (of Corollary \ref{Binjcor})}
We have proven that if $Q_r$ is the projection onto $\ran{B(rL)} ^{-\| \cdot \|}$ that $Q_r \stackrel{SOT}{\rightarrow} Q = B(L) B(L) ^*$. Consider the inner-outer factorization of $B(rL) \in \mult \otimes \C ^{1 \times N}$. Let $\{ e_k \} _{k=1} ^N$ be the standard orthonormal basis of $\C ^N$. Then $B_k  := B(L) (1 \otimes e_k )$ is an orthonormal basis for the wandering space of $\ran{B(L)}$. Let $ P_r := Q_r - R (Q_r \otimes I_d ) R^* $, $r \in (0, 1]$ be the orthogonal projection onto the wandering subspace, $\scr{W} _r$, of $ \ran{B(rL)} ^{-\| \cdot \|}$. Then, by Lemma \ref{lem:wand_proj_conv}, $P_r \stackrel{SOT}{\rightarrow} P$, where $P$ is the projection onto the wandering space of $\ran{B(L)}$. Define $\om _{r ;k } := P_r \Phi_r(B_k)$, for every $1\leq k \leq N$. Then each $\om _{r; k}$ is a (potentially zero) wandering vector in $\ran{B(rL)} ^{-\| \cdot \|}$, and since $P_r \stackrel{SOT}{\rightarrow} P$, $\Phi_r \stackrel{SOT}{\rightarrow} I$, and both nets are bounded, we have that 
$$ \om _{r; k} = P_r \Phi_r B_k \rightarrow P B_k = B_k; \quad \quad 1 \leq k \leq N. $$ (So for any fixed $k$, $\om _{r ; k} \neq 0$ for $r$ sufficiently close to $1$.) Let $\N _N := \{ 1 , 2 , \cdots , N \}$, and set $\N  _N (0) := \{ j \in \N _N | \ \om _{r;j} =0 \}$. We define a sequence of vectors in the wandering space of $\ran{B(rL)} ^{-\| \cdot \|}$ as follows: If $k \in \N  _N (0)$, so that $\om _{r ; k} =0$ we set $\Om _{r ; k} =0$. We then apply Gram-Schmidt orthogonalization to the (ordered) sequence:
$$ \left(  \om _{r; k} \right) _{k \in \N _N \sm \N _N (0) }. $$ This produces an orthonormal sequence of vectors which we label in order by the elements of $\N \sm \N _N (0)$. Combining this with the previous sequence of zero vectors indexed by $\N _N (0)$ yields the sequence $ \left(  \Om _{r ; k }  \right) _{k=1} ^N$, consisting of wandering vectors in $\scr{W} _r$ so that the non-zero elements of this sequence form an orthonormal set. (And $\Om _{r;k } =0$ if and only if $k \in \N _N (0)$.)  Note that for any fixed $k \in \{ 1 , ... , N \}$, $\Om _{r;k} $ converges to $B_k$ in Fock space norm as $r \uparrow 1$ so that for any fixed $k \in \N _N$ and $r$ sufficiently close to $1$, $\Om _{r; k} \neq 0$. Further observe, by Lemma \ref{monrank}, that the set, $\{ \om _{r ; k } \}$ has dense linear span in the wandering space of $\ran{B(rL)} ^{-\| \cdot \|}$ so that the set,
$$ \{ \Om _{r ; k } \} _{k \in \N _N \sm \N _N (0)}, $$ is an orthonormal basis of wandering vectors for $\ran{B(rL)} ^{-\| \cdot \|}$. The wandering dimension, $N_r \leq N$, of $\ran{B(rL)} ^{-\| \cdot \| }$, is then the cardinality of the set $\N _N \sm \N _N (0)$.
We then define: 
$$ \wt{\Om} _r (L) := \left( M^L _{\Om _{r; 1}} , \cdots , M^L _{\Om _{r; N}} \right) : \hardy \otimes \C ^N \rightarrow \hardy, $$ and 
$$ \Om _r (L) := \left( \Om _{r; j } (L) \right) _{j \in \N _N \sm \N _N (0) }. $$ 
Observe that each non-zero $\Om _{r ; j } (L) = M^L _{\Om _{r ;j}}$ (for $j \in \N _N \sm \N _N (0)$), is an isometric, or inner left multiplier.  It follows that $\wt{\Om } _r (L) \in \mult \otimes \C ^{1\times N}$ is a partially isometric left multiplier and $\Om _r (L) \in \mult \otimes \C ^{1\times N_r}$ is the inner left multiplier obtained from $\wt{\Om } _r (L)$ by deleting any zero entries.  The inner-outer factorization of $B(rL)$ is then 
$$ B(rL) = \Om _r (L) F_r (L), $$ where $F_r (L) := \Om _r (L) ^* B(rL)$. If $N_r < N$, we add a tail end of $N-N_r$ zeroes to $\Om _r (L)$ to obtain
$$ \hat{\Om} _r (L) := \left( \Om _r (L) , 0, \cdots , 0 \right) \in \mult \otimes \C ^{1 \times N}. $$ If we set $\hat{F} _r (L) := \hat{\Om } _r (L) ^* B(rL)$ then note that we still have
$$ B(rL) = \hat{\Om} _r (L) \hat{F} _r (L), $$ where $\hat{F} _r (L)$ is simply $F_r (L)$ with $N - N_r$ rows of zeroes added to make it `square'. In particular, since $\Om _r (L)$ is an isometry, we have that $$ \ker{B(rL)} = \ker{F_r (L)} = \ker{\hat{F} _r (L)}. $$
Observe that there is a unitary basis permutation matrix $U _r \in \C ^{N \times N}$ so that 
$$ \wt{\Om} _r (L) = \hat{\Om} _r (L) (I _{\hardy} \otimes U _r). $$ If for example, $N=3$, $N_r =2$ and 
$$ \wt{\Om} _r  = \left( \Om _{r;1} , 0 , \Om _{r;3}  \right), $$
$$ \hat{\Om } _r = \left( \Om _{r;1}, \Om _{r;3} , 0 \right), $$ then, 
$$ U _r = \bpm 1 & 0 & 0\\ 0& 0& 1  \\0 & 1 & 0  \epm, $$ satisfies
$ \hat{\Om } _r (L) (I _{\hardy} \otimes U _r ) = \wt{\Om } _r (L)$. If we then define,
\ba \wt{F} _r (L) & := & \wt{\Om} _r (L) ^* B_r (L)  \nn \\
& = & (I_{\hardy} \otimes U _r ^*) \hat{\Om} _r (L) ^* B(rL) \nn \\
& = & (I_{\hardy} \otimes U _r ^* ) \hat{F} _r (L), \nn \ea we see that 
$$ \ker{B(rL)} = \ker{\hat{F} _r (L) } = \ker{\wt{F} _r (L)}. $$ 

We claim that $\wt{\Om} _r (L)$ converges in $WOT$ to $B(L)$. Indeed, each component $\Om _{r; k}$ converges to $B_k = B_k (L ) 1$ in Fock space norm, so that $\Om _{r; k} (Z) \rightarrow B_k (Z)$ in the NC unit ball. This pointwise convergence and the uniform boundedness of the $\Om _{r ; k} (L), B_k (L)$ (these are all isometries or $0$) implies $WOT$ convergence of $\Om _{r ;k } (L)$ to $B_k (L)$ for any fixed $k$ (see for example \cite[Lemma 2.5]{SSS}). To prove that the entire row $\wt{\Om } _r (L)$ converges in $WOT$ to $B(L)$, let $\mbf{h} \in \hardy \otimes \C ^N$ and $g \in \hardy$ be any fixed vectors. Given any $\eps >0$ choose $M \in \N$ sufficiently large so that if 
$$ \mbf{h} = \bpm h_1 \\ \vdots \\ h_N \epm, \quad \mbox{then}, \quad \sum _{M+1} ^N \| h_k \| ^2 _{\hardy}  < \eps. $$ Then, 
$$ \left| \ip{ (\wt{\Om } _r (L) - B(L) ) \mbf{h} }{g}_{\hardy} \right| \leq \eps \cdot \| g \|  + 
\left|  \sum _{k=1} ^M \ip{ \left( \Om _{r,k} (L) - B_k (L) \right) h_k }{g}_{\hardy} \right|, $$ 
which can be made arbitrarily small as $r\uparrow 1$ since each $\Om _{r;k} (L)$ converges in $WOT$ to $B_k (L)$. 

Since the adjoint map is $WOT-$continuous, it then follows that $\wt{\Om} _r (L) ^* \stackrel{WOT}{\rightarrow} B(L) ^*$. Finally, since $B(rL) \stackrel{SOT}{\rightarrow} B(L)$, we obtain that 
$$ \wt{F} _r (L) = \wt{\Om} _r (L) ^* B(rL) \stackrel{WOT}{\rightarrow} I_{\hardy} \otimes I_N. $$
(Here, note that if $A_k$ and  $B_k$ are uniformly bounded nets of operators on a Hilbert space so that $A_k \stackrel{WOT}{\rightarrow} A$ and $B_k \stackrel{SOT}{\rightarrow } B$, then $A_k B_k$ converges in the $WOT$ to $AB$.)  Since $\wt{F} _r (L)$ converges in $WOT$ to $I_{\hardy} \otimes I_N$, it follows that 
\ba \ipcN{\mbf{c}}{\wt{F} _r (0) \mbf{c} '} & = &  \ip{1 \otimes \mbf{c}}{\wt{F} _r (L) (1 \otimes \mbf{c} ')} \nn \\
& \rightarrow & \ipcN{\mbf{c}}{\mbf{c}'}, \nn \ea and this proves that $\wt{F} _r (0) \in \C ^{N \times N}$ converges in $WOT$ to $I_N$. As observed previously, $\ker{B(rL)} = \ker{\wt{F} _r (L)}$ so that $B(rL) h = 0$ implies $\wt{F} _r (L) h =0$.
However, if 
\ba  \wt{F} _r (L) & = & \sum _{\alpha} L^\alpha \otimes \wt{F} _{r, \alpha}, \quad \mbox{and} \nn \\  
 h  & = & \sum _\beta L^\beta 1 \otimes h_\beta \in \ker{B(rL)}, \nn \ea  then,  
$$ 0   =   \wt{F} _r (L) h = \sum _\ga L^\ga 1 \otimes \sum _{\alpha \cdot \beta = \ga} \wt{F} _{r, \alpha} h_\beta. $$  
All coefficients must vanish, so that in particular, 
$$ \wt{F} _r (0) h(0) = 0. $$    Now given any $\mbf{c} , \mbf{c} ' \in \C ^N$, we have that 
$$ e(rL) 1 \otimes \mbf{c} \in \ker{B(rL)} = \ker{F_r (L)}  = \ker{ \wt{F} _r (L)}. $$ 
It follows that $0 = \wt{F} _r (L) e(rL) 1 \otimes \mbf{c}$, and in particular, 
$$ 0 = \wt{F} _r (0) e(0) \mbf{c}. $$ Then, 
$$ 0 = \ipcN{\mbf{c} ' }{\wt{F} _r (0) e(0) \mbf{c} } \rightarrow \ipcN{\mbf{c}'}{e(0) \mbf{c}}. $$
Since $\mbf{c} , \mbf{c} ' \in \C ^N$ were arbitrary we conclude that $e(0) = 0$. By Lemma \ref{nzerovac}, we conclude that $e \equiv 0$ vanishes identically so that $B(rL)$ is injective for $0< r \leq 1$. 
\end{proof}

\begin{thm}[NC Blaschke-Singular-Outer factorization] \label{NCBSOthm}
Any $H \in H^p (\B ^d _\N )$, $p \in \{ 2 , \infty \}$, has a unique Blaschke-Singular-Outer factorization:
$$ H = B \cdot S \cdot F; \quad \quad \quad B,S \in \mult, \ F \in H^p (\B ^d _\N ),$$ where $B$ is an NC Blaschke inner, $\mr{Sing} (B) = \mr{Sing} (H)$, $S$ is NC singular inner and $F$ is an NC outer function. The factors are unique up to constants of unit modulus.
\end{thm}
\begin{proof}
By the NC inner-outer factorization, any $H \in H^p (\B ^d _\N )$, $p=2$ or $p=\infty$, factors as $H = \Theta \cdot F$ for an NC inner $\Theta \in \mult $ and an NC outer $F \in H^p (\B ^d _\N )$ \cite[Corollary 2.2, Corollary 2.3]{DP-inv}, \cite[Theorem 2.1]{AriasPopescu}. By Proposition \ref{rowcolprop} and Corollary \ref{NCsingular}, 
$$ \Theta = B \cdot S = \bsm B_1 ,& \cdots , &  B_N \esm \bsm S_1 \\ \vdots \\ S_N \esm, $$  for a Blaschke row-inner $B$ and a column-inner $S$, both of length $N$. In Corollary \ref{idempotent}, we constructed an NC idempotent, $e$, $e (rL) \in \mult \otimes \C ^{N \times N}$ for $r \in [0, 1)$, so that $\ker{B(rL)} = \ran{e (rL)}$. Moreover, if $E (rL) = I _{\hardy } \otimes I_N - e (rL)$, then $E(rL) = S(rL) H_r (L)$, is an NC idempotent and $H_r (L) S (rL) = I_{\hardy}$, so that $H _r (L) \in \mult \otimes \C ^{1 \times N}$ is a left inverse for $S(rL)$. (Also recall that we can write $H_r (L) = H(rL)$ by Corollary \ref{idempotent} and  Remark \ref{rdepend}.) Corollary \ref{Binjcor} shows that $e \equiv 0$ so that $ S(rL) H (r L) = E (rL) = I_{\hardy} \otimes I_N$ for any fixed $0<r<1$. This means that
the diagonal components obey:
$$ S_k (rL) H_{k} (r L) = I_{\hardy} = H_{k} (rL) S_k (rL). $$ 
On the other hand, in Corollary \ref{NCsingular} we proved that $H (r L)$ is a left inverse for $S(rL)$ so that 
$$ I_{\hardy} = H (r L) S(rL) = \sum _{k=1} ^N H_{k} (r L) S_k (rL) = N \cdot I_{\hardy}. $$
This proves $N=1$, and then $S(rL)$ is an invertible left scalar multiplier with inverse $H (r L)$. In particular, the NC variety of $S$ is the empty set so that $\scr{S} (S) = \hardy$ and $S$ is an NC singular inner function.
\end{proof}
When $d=1$, we recover the classical Blaschke-Singular-Outer factorization with a new operator-theoretic proof:
\begin{cor} \label{dequals1}
Given any $h \in H^p (\D )$, $p \in \{ 2 , \infty \}$, the NC Blaschke-Singular-Outer factorization of $h$ and the classical Blaschke-Singular-Outer factorization of $h$ coincide.  That is, if $h = b \cdot s \cdot f$ is the classical Blaschke-Singular-Outer factorization of $h$, then the range of $b(M_z) = M_b$ is the singularity space of $h$.
\end{cor}
\begin{proof}
As observed in the introduction, if $h =b \cdot s \cdot f$ is the classical Blaschke-Singular-Outer factorization of $h \in H^p (\D)$, $p \in \{ 2 , \infty \}$, then 
$$ \ran{M_b} = \scr{D} (h) := \left\{ f \in H^2 \left| \ \frac{f}{h} \in \mr{Hol} (\D ) \right. \right\}, $$ is the set of all $H^2$ functions which are divisible by $h$. On the other hand, if 
$h = B \cdot S \cdot F$ is the NC Blaschke-Singular-Outer factorization of $h$ obtained by setting $d=1$ in Theorem \ref{NCBSOthm} above, then it is clear from \cite[Corollary 2.2]{DP-inv} that $F=f$, and it remains to show that 
$$ \ran{M_B} = \scr{S} (h) = \{ g \in H^2 | \ y^* g(Z) =0 \ \forall (Z,y) \in \mr{Sing} (h) \}, $$ coincides with $\ran{M_b} = \scr{D} (h)$. Clearly $g \in \scr{D} (h)$ if and only if every zero of $h$ is a zero of $g$ of greater or equal multiplicity. If $w \in \D$ is a zero of $h$ of order $n$, consider
$$ W := \bsm w & \eps &   &  \\  & \ddots & \ddots  & \\ & &  \esm \in \C ^{(n+1) \times (n+1)}, $$ where we choose $0< \eps < |w |$ so that $W$ is a strict contraction. The image of $W$ under $h$ is
$$ h(W) = \bsm h(w) & \eps h'(w) & \eps ^2 \frac{h' (w)}{2!} & \cdots & \eps ^n \frac{h ^{(n)} (w)}{n!} \\ & h(w) & \ddots  &  & \\ & & \ddots & & \\ & & & & \\ 
& & &  &  \eps h ' (w) \\
& & & &  h(w) \esm, $$ which vanishes identically as $h$ has a zero of order $n$ at $w \in \D$. It follows that for any $y \in \C ^{n+1}$, $(W, y) \in \mr{Sing} (h)$, so that any $g \in \scr{S} (h)$ is necessarily such that $g(W) \equiv 0$. This is equivalent to $w$ being a zero of $g \in H^2$ of order at least $n$, and we conclude that $\scr{S} (h) \subseteq \scr{D} (h)$ so that $\ran{M_B} \subseteq \ran{M_b}$. Conversely, if $(Z,y) \in \mr{Sing} (h)$ then, 
\ba 0 & = & y^* h(Z) \nn \\
& = & y^* b(Z) s(Z) f(Z), \nn \ea where $s(Z) f(Z)$ is invertible, by spectral mapping, since $s,f$ are non-vanishing in $\D$. This proves that $y^* b(Z) =0$ for any $(Z,y ) \in \mr{Sing} (h)$ so that $b = M_b 1 \in \scr{S} (h) = \ran{M_B}$, $b = M_B g = B g$, for some $g \in H^2$. If $p \in \C [ z ]$ is any analytic polynomial, then 
$$ M_b p = M_p b = M_p M_B g = M_B p g \in \ran{M_B}. $$  Since $M_b \C [ z ]$ is dense in $\ran{M_b}$, we conclude that $\ran{M_b} \subseteq \ran{M_B}$ so that $M_b, M_B$ have the same range. Since $b,B$ are inner functions in $\D$ with the same range, they are equal up to a unimodular constant. Without loss of generality $B=b$ and $F=f$ so that $S =s$ as well. 
\end{proof}

\subsection{The infinite level}

A natural question is whether it is really necessary to include the infinite level, $\mr{Sing} _\infty (H)$ in our definition of NC variety. Our current operator-theoretic proof of the NC Blaschke-Singular Outer factorization theorem seems to rely on this. Namely, one can define the \emph{finite NC variety}:
$$ \mr{Sing} _\N (H) := \bigsqcup _{n \in \N } \mr{Sing} _n ( H), $$ and the \emph{finite singularity space}:
$$ \scr{S} _\N (H) := \{ h \in \hardy | \ y^* h(Z) = 0 \ \forall (Z,y) \in \mr{Sing} _\N (H) \}, $$ and this is again a closed $R-$invariant subspace. Applying similar factorization arguments to those in the proof of Proposition \ref{rowcolprop} to an NC inner $H = \Theta \in \mult$ again yields: 
$$ \Theta (L) = B ' (L) S ' (L) = \bsm B_1 ' (L) , &  \cdots , & B_N ' (L) \esm \bsm S' _1 (L) \\ \vdots \\ S' _N (L) \esm, $$ for some `finite level' NC Blaschke inner row, $B'$, \emph{i.e.} $\ran{B' (L)} = \scr{S} _\N (\Theta )$, and a `finite level' NC inner column, $S'$, where $N \in \N \cup \{ \infty \}$. If $\Theta (L) = B(L) S(L)$ is the `infinite level' (scalar) NC Blaschke-Singular factorization of $\Theta$ given by Theorem \ref{NCBSOthm}, it could be that $\ran{B(L)}  =\ran{B' (L)}$, so that $B' (L) = B(L)$ up to a unimodular constant, and $B' (L) $ is scalar.  If this were the case, unrestrictedly, then there would be no need to include the infinite level in our definition of left NC variety. While we currently do not know whether or not this is the case, we can show that if $p \in \fp$ is any NC polynomial with NC Blaschke-Singular-Outer factorization $p = BSF$, then $B = B'$ is determined by the finite NC variety of $p$.

\begin{prop}
If $p \in \fp$, then any $(Z,y) \in \mr{Sing} _\infty (p)$ can be approximated by finite dimensional $(Z^{(k)}, y^{(k)}) \in \mr{Sing} _{n_k} (p)$, $n_k < \infty$, in the sense that 
$$ K \{ Z , y , v \} = wk-\lim _{k \rightarrow \infty } K \{ Z^{(k)} , y^{(k)} , v \}. $$
In particular,
$$ \scr{S} (p) = \scr{S} _\N (p). $$ 
\end{prop}
\begin{proof}
Suppose that $m$ is the homogeneous degree of $p$, and that $(Z,y) \in \mr{Sing} _\infty (p)$. Define the subspace 
$$ \mc{K} := \bigvee _{|\alpha | \leq m} (Z^\alpha ) ^* y \subseteq \C ^\infty := \ell ^2 (\N ), $$ where we assume $y \in \C ^\infty$. Define the row contaction,
$$ X _j := P_{\K} Z _j | _{\K}, $$ the compression of $Z$ to the finite dimensional subspace $\K$, and set $x := y \in \K$. We claim that $(X  , x ) \in \mr{Sing} _N (p)$ where $N := \mr{dim} (\K )$. Indeed, this is easy to verify for $m =1$. If $m >1$ then observe that 
\ba X_j   ^* X_k  ^* x  & = & P_{\mc{K}} Z_j ^* P _\K Z_k ^* x  \nn \\
& = & P_\K Z_j ^* P_\K Z_k ^* y \nn \\
& = & P_\K Z_j ^* Z_k ^* y, \nn \ea and similarly, for any $|\alpha | \leq m$, 
$$ (X ^* ) ^\alpha x = P_\K (Z ^* ) ^\alpha x = P_\K (Z^* ) ^\alpha y. $$ 
It follows that 
$$ p (X) ^* x = P_\K p (Z) ^* y =0, $$ so that $(X,x) \in \mr{Sing} _N (p)$.
For any $n \geq m $ let 
$$ \K (n) := \bigvee _{|\alpha | \leq n} (Z^\alpha ) ^* y, $$ and set $X(n) _j := P_n Z_j | _{\K (n)}, $ where $P _n : = P _{\K (n)}$. This produces a sequence of finite-dimensional singularity points
$$ (X(n) , y ) \in \mr{Sing} _{N_n} (p), $$ so that 
$$ K \{ X(n) , y , v \} \stackrel{w}{\rightarrow} K \{ Z , y , v \}.$$ Indeed, by Lemma \ref{kerform}, 
$$ K \{ X(n) , y ,v \}  =  \sum _{\alpha } \ip{X(n) ^\alpha v}{y} L^\alpha 1, $$
where, for any fixed $|\alpha | < n$,
\ba \ip{X(n) ^\alpha v}{y} & = & \ip{v}{ \left( (P_n Z P_n ) ^\alpha \right) ^* y} \nn \\
& = & \ip{Z^\alpha P_n v}{y}. \nn \ea For any fixed $\alpha \in \F ^d$, $\ip{Z^\alpha P_n v}{y} \rightarrow \ip{Z^\alpha v}{y}, $ so that 
$$ \ip{X(n) ^\alpha v}{y} \stackrel{\longrightarrow}{\mbox{\tiny $n \rightarrow \infty$}} \ip{Z^\alpha v}{y}. $$

Since each $\| X (n) \| \leq \| Z \| <1$, the NC kernel vectors $K \{ X(n), y , v \}$ are uniformly bounded in Fock space norm.  This, combined with the convergence of their coefficients implies that $K \{ X(n) , y, v \}$ converges weakly to $K\{ Z ,y , v \}$.
In particular,
$$ \scr{S} _\N (p) = \bigvee _{(Z,y) \in \mr{Sing} _\N (p)} K \{Z,y,v \} = \bigvee _{(Z,y) \in \mr{Sing}  (p)} K \{Z,y,v \} = \scr{S} (p). $$ 
\end{proof}

Another related and perhaps easier question is whether there exists an $H \in \mult$, such that $\mr{Sing}_{\N}(H) = \emptyset$, but $\mr{Sing}(H) \neq \emptyset$? A positive answer to this question, of course, implies that one cannot dispense with the infinite level. However, a negative answer does not tell us to what extent $\mr{Sing}(H)$ is determined by $\mr{Sing}_{\N}(H)$.

\section{NC Blaschke and Singular Examples} \label{sec:examples}

\subsection{Homogeneous NC polynomials and NC Blaschke inners}

In this example we will show that every homogeneous free polynomial $p \in \fp$ is a constant multiple of a Blaschke inner. Let $p \in \mult$ be a homogeneous polynomial. Since $p(L) = M^L _p$ is a constant times an isometry, we may assume without loss of generality that $p(L)$ is an isometry, \emph{i.e.} $p$ is inner. It is immediate that $\mr{Sing}(p)$ is homogeneous in the first coordinate, \emph{i.e.}, if $(Z,y) \in \mr{Sing}(p)$, then for every $\lambda \in \overline{\D}$, $(\lambda Z, y) \in \mr{Sing}(p)$. 
Let $f \in \scr{S}(p)$ and $(Z,y) \in \mr{Sing}(p)$. Write $f = \sum_{n=0}^{\infty} f_n$, the Taylor-Taylor series of $f$ at $0 \in \B ^d _1$, where $f_n$ are the homogeneous components. Then we immediately have
\[
0 = y^* \int_0^{2 \pi} e^{-i n \theta} f( e^{i \theta} Z) \frac{d \theta}{2 \pi} = y^* f_n(Z).
\]
Hence for every $n \in \N$, $f_n \in \scr{S}(p)$. By the Bergman Nullstelensatz \cite[Theorem 6.3]{HelMcC04} we have that $f_n = p g$, for some homogeneous $g$. This proves that $f$ is in the range of $p (L) $ and we conclude that $\scr{S}(p) = \ran{p(L)}$ so that $p$ is Blaschke, by definition.

\subsection{The Weyl algebra relation}
For any $w \in \D$, consider the M\"{o}bius transformation:
$$ \mu _w (z) := \frac{z-w}{1-\ov{w} z}. $$ 

\begin{lemma}
If $V \in \mc{L} (\mc{H})$ is an isometry then $\mu _w (V) $ is also an isometry.
\label{Fshift}
\end{lemma}
\begin{proof}
Consider:
$$ \mu _w (V) ^* \mu _w (V)  =  (I - w V ^* ) ^{-1} (V^* - \ov{w} ) (V - w) ( I - \ov{w} V ) ^{-1}. $$ Expand the middle term:
\ba 
(V^* - \ov{w} ) (V - w) & = & I - \ov{w} V - w V^* + |w| ^2 \nn \\
& = & (I - w V^* ) (I - \ov{w} V), \nn \ea and this proves the claim.
\end{proof}

In the classical Hardy space literature, any M\"{o}bius transformation composed with a contractive analytic function in the disk is sometimes called a \emph{Frostman shift} \cite{Frostman,Frostman2}, see also \cite[Section 2.6]{GR-model}. 

\begin{cor}{ (NC inner Frostman shifts)}
If $\Theta \in \mult$ is inner, then for any $w \in \D$,
$$ \Theta _w := \mu _w (\Theta ) = (I - \ov{w} \Theta ) ^{-1} (\Theta - wI ), $$ is also inner.
\end{cor}

The main result of this subsection will be:
\begin{thm} \label{FSBlaschke}
Let $V(Z)$ be any inner NC homogeneous polynomial. For any $w \in \D$, the NC Frostman shift $V_w (Z) = \mu _w (V(Z))$ is Blaschke. 
\end{thm}

Again, in the classical Hardy space literature, given any inner $\theta \in H^\infty$, and any $w \in \D$, there is a natural unitary (isometric and onto) multiplier, $C_w (z)$, from $(\theta H^2 ) ^\perp$ onto $(\theta _w H^2 ) ^\perp $, where as before $\theta _w = \mu _w (\theta )$ is the $w-$Frostman shift of $\theta$. The unitary multiplication operator, $M _{C_w} : (\theta H^2 ) ^\perp \rightarrow (\theta _w H^2 ) ^\perp$ is sometimes called a \emph{Crofoot Transform} \cite{Crofoot}, \cite[Theorem 6.3.1]{GR-model}.

\begin{prop}{ (NC Crofoot Transform)} \label{ncCrofoot}
Left multiplication by 
$$ C_w (Z) := \sqrt{1 - |w| ^2} \left( I_n - \ov{w} \Theta (Z) \right) ^{-1}, $$ 
is an isometry from $\left( \Theta  (L ) \hardy \right) ^\perp$ onto $\left( \Theta _w (L) \hardy \right) ^\perp$.
\end{prop}
\begin{proof}
The NC kernel for the orthogonal complement of $\ran{\Theta _w (L)}$ is
\ba & &  K^{\Theta _w} (Z,W)  =  K (Z,W) - \Theta _w (Z) K (Z,W) \Theta _w (W) ^*, \nn \\
& = & ( I - \ov{w} \Theta (Z) ) ^{-1} \cdot \nn \\
& & \underbrace{ \left( (I - \ov{w} \Theta (Z) ) K(Z,W) (I - w \Theta (W) ^* ) - (\Theta (Z) - w I) K(Z,W) (\Theta (W) ^* - \ov{w} I ) \right)}_{=: G(Z, W )} \nn \\
 & & \cdot (I - w \Theta (W) ^* ) ^{-1}. \nn \ea
The expression $G(Z,W)$ can be expanded as:
\ba & &  K(Z, W) - \ov{w} \Theta (Z) K(Z,W) - w K(Z,W) \Theta (W) ^*  + |w| ^2 \Theta (Z) K(Z,W ) \Theta (W) ^* \nn  \\ 
& & - \Theta (Z) K(Z,W) \Theta (W) ^* + w K (Z,W) \Theta (W) ^* + \ov{w} \Theta (Z) K(Z,W) - |w| ^2 K(Z,W) \nn \\
& = & (1 - |w| ^2 ) \left( K (Z,W) - \Theta (Z) K(Z,W) \Theta (W) ^* \right) \nn \\
& = & (1 - |w| ^2 ) K^\Theta (Z,W). \nn \ea 
Hence, 
$$ K^{\Theta _w} (Z,W) = (1 - |w| ^2 ) ( I - \ov{w} \Theta (Z) ) ^{-1} K^\Theta (Z,W)  (I - w \Theta (W) ^* ) ^{-1}, $$ and the claim follows readily from this formula.
\end{proof}

Let $V \in \fp$ be an inner free homogeneous polynomial of degree $n \in \N _0$, fix $w \in \D$,
and consider the operator 
$$ \left( I - V \left( \frac{\ov{w} ^{1/n}} {r} L \right) \right) ^{-1} = \left( I - \frac{\ov{w}}{r^n} V(L) \right) ^{-1}, $$ where $\ov{w} ^{1/n}$ is any $n^{\mbox{th}}$ root of $\ov{w}$, and $0<r <1$ is chosen so that $$ \frac{|w|}{r^n} < 1, \ \Rightarrow , \ \emph{i.e.} \ |w| ^{1/n} < r < 1, $$ to ensure that this operator is well-defined as a convergent geometric series.

\begin{lemma}
Given any $h \in \ran{V(L) } ^\perp = \ker{V(L) ^*}$, and $|w| ^{1/n} < r < 1$, 
$$ h ^{(r)} := \left( I - \frac{\ov{w}}{r^n} V(L) \right) ^{-1} h \in \ker{V_w (rL) ^*}. $$
\end{lemma}
\begin{proof}
Expand $h^{(r)}$ as a convergent geometric series and calculate:
\ba V (r L ) ^* h ^{(r)} & = & r^n V(L) ^* \sum _{k=0} ^\infty \frac{\ov{w} ^k}{r^{n\cdot k}} V(L) ^k h \nn \\
& = & r^n \underbrace{V(L) ^* h}_{=0} + r^n \sum _{k=1} ^\infty \frac{\ov{w} ^k}{r^{n\cdot k}} V(L) ^{k-1} h \nn \\
& = & r^n \frac{\ov{w}}{r^n} \sum _{k=1} ^\infty \frac{\ov{w} ^{k-1}}{r^{n\cdot (k-1)}} V(L) ^{k-1} h \nn \\
& = & \ov{w} h^{(r)}. \nn \ea This proves that every $h^{(r)}$ is an eigenvector of $V(rL) ^*$ to eigenvalue $\ov{w}$. It then follows that,
\ba V_w (rL) ^* h^{(r)} &= & \mu _w (V(r L)) ^* h^{(r)} \nn \\
& = & (I  - w V(rL) ) ^{-1} \underbrace{(V (rL) ^* - \ov{w} I ) h^{(r)} } _{=0}. \nn \ea  
\end{proof}
The above lemma implies that the following linear span of NC Szeg\"o kernels,
$$ \mc{K} := \bigvee _{\substack{|w| ^{1/n} < r < 1 \\ g \in \ran{V(L)} ^\perp, \ h \in \hardy}} K \{ rL , g^{(r)} , h \} \subseteq \ran{V_w (L) } ^\perp. $$   
If $B_w (L)$ is the Blaschke factor of the inner $V_w (L)$, then it follows that 
$$ \mc{K} \subseteq \ran{B_w (L) } ^\perp \subseteq \ran{V_w (L) } ^\perp. $$ 
To prove that $V_w (L)$ is Blaschke, \emph{i.e.} that $V_w = B_w$, it then suffices to show that $\mc{K} = \ran{V_w (L)} ^\perp$.

\begin{proof}{ (of Theorem \ref{FSBlaschke})}
Consider any $K \{ rL , g^{(r)}, 1 \} \in \mc{K}$, where $|w| ^{1/n} < r < 1$ is fixed, $g \in \ran{V(L)} ^\perp$, and $g^{(r)} = \left( I - \frac{\ov{w}}{r^n} V(L) \right) ^{-1} g$ as above.
This NC Szeg\"{o} kernel can be expanded as:
$$ K \{ rL , g^{(r)}, 1 \}  =  K \{ rL , g^{(r)} , 1 \} (Z) = \sum _{\alpha \in \F ^d} r^{|\alpha |} \ip{ L^\alpha 1}{g^{(r)}}_{\hardy} Z^\alpha. $$  
Further expanding each $g^{(r)}$ as a convergent geometric sum, the $\alpha ^{\mr{th}}$ Taylor series coefficient is:
\ba r^{|\alpha |} \ip{L^\alpha 1 }{g^{(r)}}_{\hardy} & = & \sum _{k=0} ^\infty \ov{w} ^k \frac{r^{|\alpha |}}{r^{n \cdot k}} \underbrace{\ip{L^\alpha 1}{V(L) ^k g } _{\hardy} }_{ =0 \ \mbox{unless } |\alpha | = n \cdot k } \nn \\
& = & \sum _{k=0} ^\infty \ov{w} ^k \ip{L^\alpha 1}{V(L) ^k g}_{\hardy} \nn \\
& = & \ip{L^\alpha 1}{\left( I - \ov{w} V(L)  \right) ^{-1} g}_{\hardy}. \nn \ea 
In the above the first line vanishes unless $|\alpha | = nk$ because $V(L )$ is a homogeneous free polynomial of degree $n$, so that any $V(L) ^k $ is a homogeneous free polynomial of degree $n \cdot k$. If $|\alpha | \neq n \cdot k$, then
$$ (V(L) ^{k} ) ^* L^\alpha 1 = 0. $$ 
This proves that 
\ba K \{ r L , g^{(r)} , 1 \} & = & \left( I - \ov{w} V(L)  \right) ^{-1} g \nn \\
& = & \sqrt{ 1 - |w| ^2 } ^{-1}  M^L _{C _w  } g, \nn \ea which belongs to $\ran{\mu _w ( V (L ) ) } ^\perp$ by Proposition \ref{ncCrofoot}. Since $g$ can be any element in $\ran{V(L)} ^\perp$, and left multiplication by the NC Crofoot multiplier $C _{w} (Z)$ is an isometry of $\ran{V(L)} ^\perp$ onto $\ran{V_w (L) } ^\perp$, it follows that the set $\mc{K}$ of NC Szeg\"{o} kernels is actually equal to 
$ \ran{V_w (L) } ^\perp$, and this proves that $V_w = B_w$ is Blaschke.
\end{proof}

\begin{eg}
The homogeneous free polynomial,
$$ V(Z) := \frac{1}{\sqrt{2}} \left( Z_1 Z_2 - Z_2 Z_1 \right), $$ is a (left) inner multiplier. Consider the following free polynomial of degree $2$:
$$ p(Z) := I_n - \sqrt{2} V(Z) = I_n - Z_1 Z_2 + Z_2 Z_1. $$ 
This free polynomial has the inner-outer factorization:
$$ p(Z) = \underbrace{\left( \frac{I_n}{\sqrt{2} } - V(Z) \right) \left( I_n - \frac{1}{\sqrt{2}} V(Z) \right) ^{-1}}_{= \mu _{\frac{1}{\sqrt{2}}} (V(Z)), \ \mbox{inner}} \cdot \underbrace{\sqrt{2} \left( I_n - \frac{V(Z)}{\sqrt{2}} \right)}_{\mbox{outer}}. $$ 
Here, setting $w = \frac{1}{\sqrt{2}}$, Lemma \ref{Fshift} implies that $\mu _w ( V(Z) ) $ is again NC inner, and the second term in the above is invertible as a left multiplier, hence NC outer. Theorem \ref{FSBlaschke} then implies that the NC inner factor, $\mu _{\frac{1}{\sqrt{2}}} (V)$, of $p$, is NC Blaschke.
\end{eg}

\subsection{Elements of the NC Disk Algebra with closed range}

\begin{thm} \label{Nosing}
If $H$ belongs to the NC disk algebra $\A := \mr{Alg} (I,L) ^{-\| \cdot \| }$ and has closed range, then its inner factor is Blaschke. In particular, any isometry in $\A$ is Blaschke.
\end{thm}

\begin{lemma} \label{bddbelow}
Given $0<r\leq 1$, the left multipliers $H_r (L) := H(rL)$ are uniformly bounded below (and hence have closed ranges) for $r$ sufficiently close to $1$.
\end{lemma}
\begin{proof}
Each of the left multipliers $H_r (L)$ are injective. By the open mapping theorem, it follows that $H_r (L)$ is bounded below if and only if it has closed range.  In particular, by assumption we have that $H(L)$ is bounded below, by say $\delta >0$. Since we further assume that $H$ is in the NC disk algebra, $H(rL) \rightarrow H(L)$ in operator norm as $r \uparrow 1$ so that there is a $0<R<1$ so that $r>R$ implies that $\| H (rL) - H(L) \| < \eps$, where $\eps := \delta /2$. Hence, for any $x \in \hardy$,
\ba \| H(rL) x \| & \geq & \| H(L) x \| - \| (H(L) - H(rL) ) x \| ^2 \nn \\
& \geq & \frac{\delta }{2} \| x \|, \nn \ea so that $H(rL)$ is uniformly bounded below by $\delta /2$ for $R < r \leq 1$.  
\end{proof}
\begin{proof}{ (of Theorem \ref{Nosing})}
To prove that $\Theta $ is Blaschke, we need to show that $\scr{S} (\Theta ) = \ran{\Theta (L)}$. Since $H$ has closed range, $\ran{H(L)} = \ran{\Theta (L)}$, where $H(L) = \Theta (L) F(L)$ is the inner-outer factorization of $H$. Hence, by Lemma \ref{innersing},
we need to show that $\scr{S} (H) = \ran{H(L)}$. 

Fix any $0<r<1$, and consider any $x \in \ran{H(rL)} ^\perp$. Observe that the pair $\left( rL , x \right) \in \mr{Sing} _\infty (H)$. It follows that if $g$ is any element in $\scr{S} (H)$, then 
$$ \ip{x}{g(rL)1}_{\hardy} =0, $$ for any $x \in \ran{H(rL)} ^\perp$, and this proves that $g_r = g(rL) 1 \in \ran{H(rL)} ^{-\| \cdot \|}, $ for any $0<r<1$. Hence, for $r$ sufficiently close to $1$,
$$ g_r = g(rL) 1 \in \ran{H (rL) }, $$ since $H(rL)$ has closed range for $r$ sufficiently close to $1$ by the previous lemma. In conclusion, 
$$ g_r = H(rL) x ^{(r)} , $$ for some $x ^{(r)}  \in \hardy$. Observe that the net $( x ^{(r)} )$ is uniformly bounded above (for $r$ close to $1$). By the previous lemma, there is an $\eps >0$ and a $0<R<1$ so that $r >R$ implies that $H(rL)$ is bounded below by $\eps$. Hence, for such $r$, since the net $(g_r )$ is convergent and hence uniformly bounded in norm,
\ba  \| g_r \| & = & \| H(rL) x  ^{(r)} \| \nn \\
&\geq & \eps \| x ^{(r)} \|, \nn \ea proving that $\| x ^{(r)} \| $ is uniformly bounded for $r >R$.
By weak compactness, there is a weakly convergent subsequence $x_k := x ^{(r_k)}$, which therefore converges pointwise to some $x \in \hardy$. Hence, for any $Z \in \B ^d _\N$,
$$ \begin{array}{ccccc}  g(r_k Z) & =&  H(r_k Z) & \cdot &  x_k (Z)  \\
                \downarrow & &  \downarrow & & \downarrow   \\
                g(Z) & = & H(Z) & \cdot &  x (Z), \end{array} $$ so that $g= H(L) x \in \ran{H(L)}$. This completes the proof.
\end{proof}

\subsection{NC singular inner examples}

If $B \in [\mult ] _1$, \emph{i.e.} $B$ belongs to the NC Schur class of all contractive NC functions in $\B ^d _\N$, then $B(L) $ is a contraction on the NC Hardy space. By \cite[Chapter 8]{NF}, (provided $B(L) \neq I _{\hardy}$) $B(L)$ is the co-generator of a $C_0$ semigroup of contractions on $\hardy$. Namely, if 
$$ H_B (L) := (I - B(L) ) ^{-1} ( I + B(L) ), $$ is the inverse Cayley transform of $B$, then $H_B (L)$ is a closed, densely-defined accretive operator (numerical range in the right half-plane), so that $H_B (Z)$ belongs to the NC Herglotz class of locally bounded (holomorphic) NC functions in $\B ^d _\N$ with positive semi-definite real part:
$$ \re{H_B (Z)} \geq 0_n, \quad \quad Z \in \B ^d _n. $$ 
Since $1$ is not an eigenvalue of $B(L)$, \cite[Theorem III.8.1]{NF} implies that 
$$ B_t (L) := \exp ( -t  H_B (L) ); \quad \quad t \geq 0, $$ is a $SOT-$continuous one-parameter monoid of contractions on $\hardy$, so that $B_t (Z) \in [ \mult ] _1$ belongs to the NC Schur class for every $t \geq 0$. Moreover, by \cite[Proposition III.8.2]{NF}, $B_t (L)$ will be an isometry on $\hardy$ for every $t \geq 0$, \emph{i.e.} $B_t $ will be NC inner, if and only if $B(L)$ is NC inner. It further follows that if $B(L)$ is NC inner, then every $B_t (L)$ will be an NC singular inner since  $$ B_t (Z) = \exp ( - H_B (Z) ), \quad Z \in \B ^d _n, $$ is clearly pointwise invertible in $\B ^d _{\aleph _0}$. This provides a large class of examples of NC singular inner functions, and products of such NC singular inner functions are again NC singular inner. It is unclear whether or not all NC singular inners can be obtained in this way.

\section{Outlook}
The NC Blaschke-Singular-Outer factorization raises several natural questions. Classically, the inner factor of any polynomial in $\D$ is a finite Blaschke product, and hence a rational analytic function with poles outside of the open disk. Rational functions have been studied extensively in the NC setting by several authors \cite{Volcic,Volcic2,Volcic3,KVV-rational,PV-realize,Helton-rational}.

\begin{quest}
If $p \in \fp$ is any NC polynomial, is its NC inner factor Blaschke?  Is it an NC rational function? Is the NC outer factor an NC polynomial?
\end{quest}

Frostman's theorem states that given any inner function, $\theta$, in the unit disk, `almost all' of its M\"{o}bius transformations are Blaschke inner. There is also a theory of so-called \emph{indestructible Blaschke products}, these are Blaschke inner functions so that their images under any M\"{o}bius transformation are again Blaschke products. In particular, the Blaschke inner factor of any polynomial (a finite Blaschke product) is indestructible  \cite{Ross-indestruct}, \cite[Frostman's Theorem, Theorem 2.6.1]{GR-model}.

\begin{quest}
Does an NC analogue of Frostman's theorem hold? If the inner factor of any NC polynomial is Blaschke, is it indestructible?
\end{quest}

Any Blaschke inner in the disk is a (potentially) infinite product of \emph{Blaschke factors}:
$$ B_w (z) := \frac{z-w}{1-\ov{w}z}. $$ Similarly one could define \emph{NC Blaschke factors} as irreducible NC Blaschke inner functions, $B$, with the property that there are no non-trivial NC Blaschke inners $B_1, B_2$ so that $B = B_1 B_2$. A final question is whether there is a nice characterization of NC Blaschke factors.

\appendix
\section{Idempotents in $\mult \otimes \C^{n\times n}$} \label{sec:extra}
\label{appendix}

\begin{thm}
Let $E \in \mult \otimes \C ^{n \times n}$ be an idempotent, then there exists an orthogonal projection $P \in \C ^{n\times n}$ and an $S \in \operatorname{GL}_n( \mult )$, such that $$E = S^{-1} \left(I_{\hardy} \otimes P\right) S.$$
\end{thm}
In particular, this implies that there are no non-trivial finitely generated projective modules over $\mult$ and thus $\mult$ is a semi-free ideal ring, see \cite[Section 2.3]{Cohn}.
\begin{proof}
Let $\cM = \ran{E} = \ker{I - E}$ and $\cK = \ker E$ and note that $\cM + \cN = \hardy\otimes \C^n$. In particular, the Friedrichs angle between $\cM$ and $\cK$ is non-zero. Additionally, the spaces $\cM$ and $\cK$ are $R \otimes I_n-$invariant and closed. Let $\scr{W}_{\cM}$ and $\scr{W}_{\cK}$ be the wandering subspaces of $\cM$ and $\cK$, respectively. Note that since $\hardy \otimes \C^n$ surjects onto $\cM$ and $\cK$, that $m = \dim W_{\cM}, k = \dim W_{\cK} \leq n$. (This follows as in the proof of Lemma \ref{monrank}.)  Let $V_{\cM} (L)  \colon F ^2 _d \otimes \C^m \to \hardy \otimes \C^n$ be the inner left multiplier in $\mult \otimes \C ^{n\times m}$  with image $\cM$ and similarly $V_{\cK} (L) \in \mult \otimes \C ^{n \times k}$ be the isometric left multiplier with image $\cK$. Consider $S (L) \in \mult \otimes  \C ^{n \times (k+m)}$ given by $S = (V_{\cM}, V_{\cK})$. Clearly, $S $ is surjective and bounded. Furthermore, since $\cM \cap \cK = \{0\}$, $S$ is also injective and thus has a bounded inverse. For every $1 \leq i \leq d$, $S (R_i \otimes I_n) = (R_i \otimes I_{m}) S$ so that $S=S(L)$. Multiplying by $S^{-1}$ on both left and right we get that $(R_i \otimes I_n) S^{-1} = S^{-1} (R_i \otimes I_{m})$. Thus $S^{-1} \in \mult \otimes \C ^{(k+m) \times n}$. 

Note that $S (L) \colon \hardy \otimes \C^{k+m} \to \hardy \otimes \C^n$ is surjective and thus $m+ k \geq n$. Similarly $S^{-1}$ is surjective and thus $n \geq m +k$. Therefore, $m+k=n$ and thus the matrix $S (L)$ is square and $E (L)$ is similar to the projection onto the $m$ last components of $\hardy \otimes \C^{m+k}$ via $S$.
\end{proof}

\begin{remark}
The similarity, $S (L)$, is not unique. Multiplication by any constant invertible matrix in the commutant of $P$, for example, will result in a different $S$.
\end{remark}
\begin{cor}
An operator-valued left multiplier $S \in  \mult \otimes \C ^{n\times k}$ is invertible if and only if $n=k$ and its inverse is in $\mult \otimes \C ^{n\times n}$.
\end{cor}

\footnotesize

\end{document}